\documentclass[12pt]{article}
\title{A Direct Algorithm to Compute the Topological Euler Characteristic and Chern-Schwartz-MacPherson Class of Projective Complete Intersection Varieties}
\author{Martin Helmer$^\dagger$  \\ \normalsize
        Appendix:  Martin Helmer$^\dagger$ and \'Eric Schost$^{\dagger\dagger}$ \\ \tiny
             $^\dagger$Department of Applied Mathematics, University of Western Ontario,  London, Canada. \texttt{martin.helmer2@gmail.com} \vspace{-2mm}\\\tiny $^{\dagger\dagger}$ Cheriton School of Computer Science, University of Waterloo, Waterloo, Canada. \texttt{eric.schost@gmail.com }}   
\date{\today}

\usepackage{appendix}
\usepackage{chngcntr}
\counterwithin{table}{section}
\usepackage{booktabs}% http://ctan.org/pkg/booktabs
\usepackage{colortbl}% http://ctan.org/pkg/colortbl
\usepackage{xcolor}% http://ctan.org/pkg/xcolor
\usepackage{graphicx}% http://ctan.org/pkg/graphicx
 \definecolor{Ftitle}{RGB}{11,46,108}
\definecolor{line}{RGB}{87,39,117}
\definecolor{chameleongreen2}{HTML}{5B1280} %purpule
\colorlet{tableheadcolor}{Ftitle!25} % Table header colour = 25% gray
\colorlet{tablerowcolor}{gray!10} % Table row separator colour = 10% gray
\usepackage[hidelinks]{hyperref}

\newcommand{\codim}{\mathrm{codim} }

\newcommand{\Proj}{\mathrm{Proj}}

\newcommand{\ZZ}{ \mathbb{Z}}
\newcommand{\oo}{\mathcal{O}}
\newcommand{\pp}{\mathbb{P}}

\usepackage{amsmath,amsthm}
\usepackage{amsfonts}

\newtheorem{algorithm}{Algorithm}
\numberwithin{algorithm}{section} 
\newtheorem{theorem}{Theorem}[section]
\newtheorem{propn}[theorem]{Proposition}
\newtheorem{corr}[theorem]{Corollary}
\newtheorem{lemma}[theorem]{Lemma}
\newtheorem{defn}[theorem]{Definition}
\newtheorem{example}[theorem]{Example}
\newtheorem{remark}[theorem]{Remark}

\begin{document}
\maketitle
\begin{abstract} \noindent
Let $V$ be a possibly singular scheme-theoretic complete intersection subscheme of $\pp^n$ over an algebraically closed field of characteristic zero. Using a recent result of \textit{Fullwood (``On Milnor classes via invariants of singular subschemes'', Journal of Singularities)} we develop an algorithm to compute the Chern-Schwartz-MacPherson class and Euler characteristic of $V$. This algorithm complements existing algorithms by providing performance improvements in the computation of the Chern-Schwartz-MacPherson class and Euler characteristic for certain types of complete intersection subschemes of $\pp^n$. %We also give a new algorithm to compute the projective degree using the so-called mixed multiplicities which can be applied as part of an algorithm to compute $c_{SM}$ classes and Euler characteristics.      

\end{abstract}
%{\bf Keywords:} Euler characteristic, Chern-Schwartz-MacPherson class, Segre class, computer algebra, computational intersection theory.
\section{Introduction}
Beginning with Euler's Polyhedral Formula (circa 1750) the Euler characteristic has developed into an important invariant for the study of topology and geometry in a wide variety of settings. In addition to providing a mechanism to enable the classification of orientable surfaces, the Euler characteristic is an important component in many results in geometry. More recently several authors have noted applications of the Euler characteristic of projective varieties to problems in statistics and physics. Specifically the Euler characteristic is used for problems of maximum likelihood estimation in algebraic statistics by Huh in \cite{huh2012maximum} as well as in string theory by Aluffi and Esole in \cite{aluffi2010new} and by Collinucci, Denef, and Esole in \cite{collinucci2009d}. 

Let $V$ be a subscheme of a projective space $\pp^n$ (over an algebraically closed field of characteristic zero $k$). One of the first computational approaches to the calculate the Euler characteristic of $V$, $\chi(V)$, was to do so by computing Hodge numbers and using the fact that the Euler characteristic is an alternating sum of Hodge numbers. This approach is implemented in Macaulay2 \cite{M2} as the function euler, where the Hodge numbers are found by computing the ranks of the appropriate cohomology rings. This approach, however, has significant drawbacks in both applicability and performance. Specifically, this method is only applicable for smooth subschemes and the computation of Hodge numbers is computationally expensive.%, additionally many Hodge numbers must be calculated to compute one Euler characteristic.  

Alternatively, one may obtain the Euler characteristic of $V$ directly from the Chern-Schwartz-MacPherson class of $V$, $c_{SM}(V)$. In particular, when we consider $c_{SM}(V)$ as an element of the Chow ring of $\pp^n$, $A^*(\pp^n)$, we have that $\chi(V)$ is equal to the zero dimensional component of $c_{SM}(V)$. This is the method we shall use to obtain the Euler characteristic. This technique has been used by several authors (e.g.\ \cite{aluffi2003computing}, \cite{jost2013algorithm}, \cite{helmer2014algorithm}) to construct different algorithms which are capable of calculating Euler characteristics of complex projective varieties. These previous methods will be discussed below. 

In addition to containing the Euler characteristic, $c_{SM}$ classes are an important invariant in algebraic geometry, providing a generalization of the Chern class to singular schemes. While there are several other generalizations of the Chern class to singular schemes (i.e.\ the Chern-Fulton and Chern-Fulton-Johnson classes, see \cite{aluffi2005characteristic} for a discussion of these), the $c_{SM}$ class is the only generalization which preserves the relation between Chern classes and the Euler characteristic. Additionally the $c_{SM}$ class has unique functorial properties (see Def.\ \ref{defn:csm_natural_transform}) and relationships to other common invariants. The $c_{SM}$ class has also found direct applications to problems from string theory in physics, see for example Aluffi and Esole \cite{aluffi2009cherntadpole}. 

%We note that throughout this document we will frequently abuse notation and let $c_{SM}(V)$ and $s(V,\pp^n)$ denote the pushforwards of the Chern-Schwartz-MacPherson and Segre classes to $\pp^n$ (respectively). We do this since by definition these characteristic classes are elements of $ A^*(V)$, however it will be more notationally convenient to consider their pushforwards to $A^*(\pp^n)$.

The Chow ring of the projective space $\pp^n$ may be expressed as the quotient ring $A^*(\pp^n)=\ZZ[h]/(h^{n+1})$ where $h$ is the rational equivalence class of a general hyperplane in $\pp^n$. Consider the hypersurface $V(f) \subset \pp^n$ defined by the homogeneous polynomial $f$. All previous methods to compute $c_{SM}(V(f))$ employ Theorem 2.1 of Aluffi \cite{aluffi2003computing}, which may be expressed as \begin{equation}
c_{SM}(V(f))=(1+h)^{n+1}-\sum_{j=0}^{n}g_j (-h)^j(1+h)^{n-j} \mathrm{\; in \;} A^*(\pp^n) \cong \ZZ[h]/(h^{n+1}), \label{eq:csm_hyper}
\end{equation} where $g_i \in \ZZ$ are integers which may be understood several different ways. In fact, the differences between these methods to compute $c_{SM}$ classes lay in how the $g_j$'s are understood and computed. The first algorithm to compute $c_{SM}(V(f))$ was that of Aluffi \cite{aluffi2003computing}. To compute the $g_j$'s this algorithm requires the computation of the blowup of $\pp^n$ along the singularity subscheme of $V(f)$ (that is the scheme defined by the partial derivatives of $f$). Hence the cost of computing the $c_{SM}$ class of a hypersurface using the method of Aluffi is that of computing the Rees algebra of the ideal defining the singularity subscheme of the hypersurface. This can be a quite expensive operation, making this algorithm impractical for many examples. 

Another algorithm to compute the $c_{SM}$ class of a hypersurface was given by Jost in \cite{jost2013algorithm}. This method makes use of Fulton's residual intersection theorem (Theorem 9.2 of Fulton \cite{fulton}) which allows Jost to consider the $g_j$'s in (\ref{eq:csm_hyper}) as the degrees of Fulton's residual scheme. Jost also shows that in the context of $c_{SM}$ (and Segre) class computations these residual schemes can be computed by finding a particular saturation. Hence the computation of the saturation to find the residual scheme and the computation of its degree are the main costs of Jost's algorithm. The algorithm of Jost is probabilistic and yields the correct result for a choice of objects lying in an open dense Zariski set of the corresponding parameter space, see Jost \cite{jost2013algorithm} or Eklund, Jost, and Peterson \cite{Jost}.

In \cite{helmer2014algorithm}, the author of this note considers the $g_j$'s as the projective degrees of a rational map defined by the partial derivatives of $f$ and gives a method to compute these projective degrees by finding the degree of a certain zero dimensional ideal (see Theorem \ref{theorem:projective_deg_theorem} below). The method given in \cite{helmer2014algorithm} to compute the projective degrees is probabilistic and yields the correct result for a choice of objects lying in an open dense Zariski set of the corresponding parameter space. This method is implemented in \cite{helmer2014algorithm} using both Gr\"{o}bner bases methods and polynomial homotopy continuation (via Bertini \cite{Bertini} and PHCpack \cite{PHCpack}); it provides a performance improvement over previous methods in many cases. A detailed comparison of these methods can be found in \cite{helmer2014algorithm}. 

For $V$ a possibly singular subscheme of $\pp^n$ all these methods require the use of the inclusion-exclusion property of $c_{SM}$ classes when $V$ has codimension higher than one. Specifically for $V_1,V_2 $ subschemes of $\pp^n$ the inclusion-exclusion property for $c_{SM}$ classes states \begin{equation}
c_{SM}(V_1 \cap V_2)=c_{SM}(V_1)+c_{SM}(V_2)-c_{SM}(V_1 \cup V_2).
\label{eq:csm_inclusion_exclusion}
\end{equation} From this we may directly deduce the following. \begin{propn}
Let $	V$ be a subscheme of $\pp^n$. Write the polynomials defining $V$ as $F=(f_1,\dots , f_m)$ and let $F_{ \left\lbrace S \right\rbrace } = \prod_{i \in S} f_i $ for $S \subset \left\lbrace 1, \dots , m\right\rbrace$. Then, $$
c_{SM}(V)= \sum_{S \subset \left\lbrace 1, \dots , m\right\rbrace} (-1)^{|S|+1}c_{SM} \left(V( F_{\left\lbrace S \right\rbrace } )\right),
$$  where $|S|$ denotes the cardinality of the integer set $S$. \label{propn:csm_higher_codim}
\end{propn}
While the use of this property allows for the computation of $c_{SM}(V)$ for $V$ of any codimension, it requires exponentially many $c_{SM}$ computations relative to the number of generators of $I$. Additionally some of the schemes considered while performing inclusion-exclusion may have significantly higher degree than the original scheme $V$. 

Below we discuss an algorithm that will allow for the direct computation of the $c_{SM}$ classes of arbitrary, possibly singular, globally complete intersection subschemes of $\pp^n$ defined by a homogeneous polynomial ideal $ I=(f_0,\dots,f_m)$ where the scheme defined by $(f_0,\dots ,f_{m-1})$ is smooth (allowing for a possible rearrangement of the generators of $I$). We also give an extension of this method to all globally complete intersection subschemes of $\pp^n$ via a form of the inclusion-exclusion property of $c_{SM}$ classes which considered only the generators of $I$ which define a singular subscheme of $\pp^n$. This new method can be implemented symbolically using Gr\"{o}bner bases methods or numerically using polynomial homotopy continuation via a package such as Bertini \cite{Bertini}. We see that this new method complements existing methods for computing $c_{SM}$ classes by providing performance improvements, particularly when the input ideal has relatively few generators which define singular schemes.   

In Section \ref{section:background} we review several important definitions which will be used throughout this note. In particular we define several different characteristics classes including the Segre and Chern-Schwartz-MacPherson classes and explore more closely the relationships between the $c_{SM}$ class and the Euler characteristic using a recent result of Aluffi \cite{aluffi2013euler}. %We also define the projective degrees and in Proposition \ref{propn:projective_deg_via_mixed_mult} we give the main component of a new algorithm to compute the projective degrees of a rational map using the notion of mixed multiplicities. 

In Section \ref{section:algorithm} we give a new expression for the $c_{SM}$ class of a complete intersection subscheme $V(f_0,\dots,f_m)$ of $\pp^n$ such that $V(f_0,\dots,f_{m-1})$ is smooth in Theorem \ref{theorem:MainTheorem}. This result is based on an expression for the Milnor class of a scheme of this type due to Fullwood \cite{fullwood2014milnor}. This expression allows us to state an algorithm to compute the $c_{SM}(V)$ for a complete intersection $V$ in $\pp^n$. This new algorithm offers performance improvements over the standard inclusion-exclusion method when only a few of the generators of the ideal defining the scheme $V$ are singular. We give some running time results for this method in Table \ref{table:csmResults} and Table \ref{table:csmResultsFiniteFeild}. The algorithms resulting from Theorem \ref{theorem:MainTheorem} are presented in Algorithms \ref{algorithm:mainAlg} and \ref{algorithm:hybrid}. 

We note that Algorithms \ref{algorithm:mainAlg} and \ref{algorithm:hybrid} constructed below are probabilistic algorithms. This is due to the fact that these algorithms make use of Algorithm 2 of the author \cite{helmer2014algorithm} which is a probabilistic algorithm to compute the Segre class of a subscheme of projective space. In \S\ref{subsection:probabilityTheoretic} we obtain a theoretical bound on the probability that Algorithm 2 of \cite{helmer2014algorithm} (and hence Algorithm \ref{algorithm:mainAlg} below) will give a correct result when working over a finite parameter space. In \S\ref{subsection:probExper} we consider the probability of success for the Segre class computations using Algorithm 2 of \cite{helmer2014algorithm} (and hence the $c_{SM}$ class computations of Algorithm \ref{algorithm:mainAlg} below) when preformed on several examples using our test computation environment.  

The Macaulay2 \cite{M2} and Sage \cite{sage} implementations of Algorithms \ref{algorithm:mainAlg} and \ref{algorithm:hybrid} (as well as the implementations of Algorithms 1, 2 and 3 of the author \cite{helmer2014algorithm}) used for testing in \S\ref{subsection:Runing_time_comparison} can be found online at \url{https://github.com/Martin-Helmer/char-class-calc}. 

%The topological Euler characteristic is an important invariant in a wide variety of areas of mathematics and has been studied by numerous authors in many different contexts. Here we explore an improved algorithm to compute the Chern-Schwartz-MacPherson class ($c_{SM}$ class) and Euler characteristic of a certain class of projective varieties. In \cite{helmer2014algorithm} the author gives a method to compute the Chern-Schwartz-MacPherson class of a projective variety $V$ using the projective degree of a rational map defined by the polynomials of the ideal corresponding to $V$. This method complements several previous methods to compute the $c_{SM}$ class, such as the method of Jost \cite{jost2013algorithm} and that of Aluffi \cite{aluffi2003computing}. All of these methods use the so called inclusion-exclusion property to compute $c_{SM}(V)$ for variates having codimesnion higher than $1$ in $\pp^n$ by reducing the computation to a computation of $c_{SM}$ classes for a collection of hypersurfaces. This is due to the fact that there is no known direct formula to compute the $c_{SM}$ classes in higher codimension for an arbitrary variety. 

\section{Background}
\label{section:background}
In this section we review the definitions of the characteristic classes that will be needed to describe the algorithm presented in Section \ref{section:algorithm}. In particular we give the definition of the Chern-Shwartz-MacPherson class in Definition \ref{defn:csm_natural_transform} and discuss its relationship with the Euler characteristic. We also define the Segre class in (\ref{eq:segre_def}), and the Chern-Fulton-Johnson class in (\ref{eq:fulton_johnson_def}). 

The algorithm given in Section \ref{section:algorithm} will rely on an expression due to Aluffi \cite{aluffi2003computing} for the Segre class in terms of the projective degrees of a rational map. We give this relation in Proposition \ref{prop:SegreProjectiveDegreeAluffi} and give the definition of the projective degrees of a rational map in \eqref{eq:projective_degrees_def}. In Theorem \ref{theorem:projective_deg_theorem} we give a result of the author's \cite{helmer2014algorithm} which provides a means to compute the projective degrees using a computer algebra system. %We will conclude this section with a new algorithm to find the projective degrees using the notion of mixed multiplicity (see Def. \ref{defn:mixed_multiplicity}); the main tool for this algorithm is given in Proposition \ref{propn:projective_deg_via_mixed_mult}. 

All characteristics classes considered here will be understood to be elements of some Chow ring. We will express the Chow ring of a $n$-dimensional nonsingular variety $M$ as $$A^*(M)=\bigoplus_{i=0}^n A^{i}(M),$$ where $A^{\ell}(M)$ is the Chow group of $M$ having codimension $\ell$ in $M$, that is $A^{\ell}(M) $ is the group of codimension $\ell$-cycles modulo rational equivalence. Where convenient we will also write $A_{j}(M)$ for the Chow group of dimension $j$, that is the group of dimension $j$-cycles modulo rational equivalence. All computations of characteristic classes will take place in the Chow ring of $\pp^n$, $A^{*}(\pp^n)$. Recall that $A^*(\pp^n)\cong \ZZ[h]/(h^{n+1})$ where $h=c_1\left( \oo_{\pp^n}(1) \right)$ is the rational equivalence class of a hyperplane in $\pp^n$ ($c_1$ denotes the first Chern class), so that a hypersurface $W$ of degree $d$ will be represented by $[W]=d\cdot h$ in $A^*(\pp^n)$. For more details see Fulton \cite{fulton}.

Given $V$ a proper closed subscheme of a variety $W$, the Segre class of $V$ in $W$ may be expressed as \begin{equation}
s(V,W)= \sum_{j\geq 1}(-1)^{j-1}\eta_*(\tilde{V}^j), \label{eq:segre_def}
\end{equation}
 where $\tilde{V}$ is the exceptional divisor of the blow-up of $W$ along $V$, $\eta: \tilde{V} \to V$ is the projection and the class $\tilde{V}^k$ is the $k$-th self intersection of $\tilde{V}$. Thus $\eta_*(\tilde{V}^j)$ denotes the pushforward of the $j$-th self intersection class of the exceptional divisor $\tilde{V}$ in the Chow ring $A^*(\tilde{V})$ to the Chow ring $A^*(V)$. For a more detailed description, see Fulton \cite[\S 4.2.2]{fulton}. We also note that in all cases considered here we will have $W=\pp^n$, allowing us to use the more concrete expression for the Segre class given in Proposition \ref{prop:SegreProjectiveDegreeAluffi}.

For a smooth scheme $X $ let $T_X$ denote the tangent bundle to $X$. For a vector bundle $E$ on $X$ let $c(E)$ denote the total Chern class of $E$, see Fulton \cite[\S 3.2]{fulton}. We will write $c(X)=c(T_X) \cap [X]$ for the total Chern class of $X$ in the Chow ring of $X$, $A^*(X)$. As a consequence of the Hirzebruch-Riemann-Roch theorem, we have that the degree of the zero dimensional component of the total Chern class of a smooth projective variety is equal to the Euler characteristic, that is \begin{equation}
\int c(T_X) \cap [X]=\chi(X). \label{eq:chern_euler_non_singular}
\end{equation}Here $\int \alpha$ denotes the degree of the zero dimensional component of the class $\alpha \in A^*(X)$, i.e.\ the degree of the part of $\alpha$ in the dimension zero Chow group $A_0(X)$. Note that we will frequently abuse notation and, given a scheme $V$ in $\pp^n$ we will write $c(V)$, $s(V,\pp^n)$ and $c_{SM}(V)$ for the pushforwards to $\pp^n$ of each characteristic class, i.e.\ we will consider the various characteristic classes as their pushforwards in $A^*(\pp^n)$ rather than in $A^*(V)$.

There exist several different generalizations of the total Chern class to singular schemes and all of these notions agree with $c(T_V) \cap [V]$ for nonsingular $V$. The Chern-Schwartz-Macpherson class is, however, unique in the sense that it is the only generalization which satisfies a property analogous to (\ref{eq:chern_euler_non_singular}) for any $V$, i.e. \begin{equation}
\int c_{SM}(V)=\chi(V). \label{eq:csm_euler}
\end{equation} A recent result of Aluffi \cite{aluffi2013euler}, which we illustrate in Example \ref{ex:AluffiInvolution}, shows that the $c_{SM}$ class has an even stronger relation to the Euler characteristic in the case of projective varieties. 

We now briefly review the definition of the $c_{SM}$ class, given in the manner of MacPherson \cite{macpherson1974chern}. For a scheme $V$, denote by $\mathcal{C}(V )$ the abelian group of finite linear combinations $\sum_W m_W \mathfrak{1}_W$, with the $W$ being (closed) subvarieties of $V$, and $m_W \in \ZZ$; $\mathfrak{1}_W$ denotes the function that is $1$ in $W$, and $0$ outside of $W$. We refer to elements $f\in \mathcal{C}(V )$ as constructible functions and write $\mathcal{C}(V )$ for the group of constructible functions on $V$. $\mathcal{C}$ can be turned into a functor by letting $\mathcal{C}$ map a scheme $V$ to the group of constructible functions on $V$ and map a proper morphism $f: V_1 \to V_2$  to $$\mathcal{C}(f)(\mathfrak{1}_W)(p)=\chi(f^{-1}(p) \cap W), \;\;\; W \subset V_1, \; p\in V_2 \; \mathrm{a \; closed \; point}.$$ 

The Chow group functor $\mathcal{A}_*$ is also a functor from algebraic varieties to Abelian groups. The $c_{SM}$ class may be realized as a natural transformation between these two functors.   
\begin{defn}
The Chern-Schwartz-MacPherson class is the unique natural transformation between the constructible function functor and the Chow group functor, that is $c_{SM}: \mathcal{C}\to \mathcal{A}_*$ is the unique natural transformation satisfying: \begin{itemize}
\item (\textit{Normalization}) $ c_{SM}(\mathfrak{1}_V)=c(T_V) \cap [V] $ for $V $ non-singular  and complete.
\item (\textit{Naturality}) $f_{*}(c_{SM}(\phi))=c_{SM}(\mathcal{C}(f)(\phi))$, for $f:X \to Y$ a proper transform of projective varieties, $\phi$ a constructible function on $X$. 
\end{itemize} \label{defn:csm_natural_transform}
\end{defn}
For a scheme $V$ let $V_{red}$ denote the support of $V$. The notation $c_{SM}(V)$ is taken to mean $c_{SM}(\mathfrak{1}_V)$ and hence, since $\mathfrak{1}_V=\mathfrak{1}_{V_{red}} $, we have $c_{SM}(V)=c_{SM}(V_{red})$. 

When $V $ is a subscheme of $\pp^n$ the class $c_{SM}(V)$ can, in a sense, be thought of as a more refined version of the Euler characteristic since it in fact contains the Euler characteristics of $V$ and those of general linear sections of $V$ for each codimension. Specifically, if $\dim(V)=m$, starting from $c_{SM}(V)$ we may directly obtain the list of invariants $$\chi(V),\chi(V \cap L_1), \chi(V \cap L_1\cap L_2), \dots,\chi(V \cap L_1\cap \cdots \cap  L_m) $$ where $L_1,\dots, L_m$ are general hyperplanes. Conversely from the list of Euler characteristics above we could obtain $c_{SM}(V)$, i.e. there exists an involution between the Euler characteristics of general linear sections and the $c_{SM}$ class in this setting. This relationship is given explicitly in Theorem 1.1 of Aluffi \cite{aluffi2013euler}; we give an example of this below. 
%\begin{theorem}[Theorem 1.1 Aluffi \cite{aluffi2013euler}] Let $V$ be any locally closed set in $\pp^n$. Let $V_r=V \cap L_1 \cap \dots \cap L_r$ be the intersection of $V$ with $r$ general hyperplanes. Define the polynomial having degree at most $n$ specified by $$\chi_V(t):= \sum_{r\geq 0} \chi(V_r)\cdot (-t)^r.$$ Define another polynomial of degree at most $n$ given by $$\gamma_V(t):=\sum_{r\geq 0} \gamma_r\cdot (-t)^r$$ here $\gamma_r=c_{SM}(V)_r$ is the coefficent of the dimension $r$ componet of $c_{SM}(V)$, that is; the polynomial $\gamma_V(t)$ is obtained by replacing $[\pp^r]\cong h^{n-r}$ with $t^r$ in $c_{SM}(V)$. Also define the map $\mathcal{I}$ specifed by $$ p(t) \mapsto \mathcal{I}(p):= \frac{t \cdot p(-t-1)+p(0)}{t+1}.$$ Then $\mathcal{I}$ is an involution and we have:\begin{equation}
%\chi_V(t)= \mathcal{I}\left( \gamma_V(t) \right), \;  \; \gamma_V(t) = \mathcal{I}\left( \chi_V(t) \right).
%\end{equation}
%\label{theorem:aluffi_involutionTheorem}
%\end{theorem}

\begin{example} 
Consider $V=V(x_0x_3-x_1x_2)$ in $\pp^3=\Proj(k[x_0,\dots,x_3])$ which is the variety defined by image of the Segre embedding $\pp^1 \times \pp^1 \to \pp^3$. We may compute $c_{SM}(V)=4h^3+4h^2+2h$ and obtain the Euler characteristics of the general linear sections using an involution formula given by Aluffi in \cite{aluffi2013euler}, specifically:
\begin{itemize}
\item First consider the polynomial $p(t)= 4+4t+2t^2 \in \ZZ[t]/(t^4) $ given by the coefficients of the $c_{SM}$ class above.
\item Next apply Aluffi's involution $$p(t) \mapsto \mathcal{I}(p):= \frac{t \cdot p(-t-1)+p(0)}{t+1} =2t^2-2t+4.$$ 
\end{itemize} This gives $\chi(V)=4,$ $\chi(V\cap L_1)=2,$ and $\chi(V\cap L_1 \cap L_2)=2.$ \label{ex:AluffiInvolution}
\end{example}

We will also make use of another generalization of the total Chern class to singular schemes called the Chern-Fulton-Johnson class and denoted $c_{FJ}$. For simplicity we will give the definition of $c_{FJ}$ only for the case where $X$ is a closed locally complete intersection subscheme of a smooth ambient variety $M$, since this will be sufficient for our purposes in this note. For a complete definition and an excellent discussion of the Chern-Fulton-Johnson classes and other related notions see Aluffi \cite{aluffi2005characteristic}. Let $X$ be a closed locally complete intersection subscheme of a smooth ambient variety $M$ and let $T_M$ denote the tangent bundle of $M$, define 
\begin{equation}
c_{FJ}(X)=c(T_M) \cap s(X,M) \label{eq:fulton_johnson_def}.
\end{equation} Also note that since we assume that $X$ is a locally complete intersection (meaning there exists a regular embedding $i:X\to M$) then by Proposition 4.1 of Fulton \cite{fulton} we have $$
c_{FJ}(X)=c(T_M) \cap s(X,M)= c(T_M) \cap \left( c(N_XM)^{-1} \cap [X] \right).
$$ Here $N_XM$ is the normal bundle to $X$ in $M$ (that is the vector bundle with sheaf of sections $\left(\mathcal{I}/\mathcal{I}^2 \right)$ where $\mathcal{I}$ is the ideal sheaf of $X$). Finally, let $V$ be a subscheme of $M$; we define the Milnor class of $V$ as \begin{equation}
\mathcal{M}(V)=(-1)^{\codim(V)}(c_{FJ}(V)-c_{SM}(V)). \label{eq:milnor_class_def}
\end{equation} Note that other sign conventions may be used in definition of the Milnor class, we use the sign convention used by \cite{fullwood2014milnor}, see Fullwood \cite{fullwood2014milnor} or Aluffi \cite{aluffi2005characteristic} for more details. 

All algorithms considered in this note will make use of the so-called projective degrees of a rational map to compute characteristics classes. We recall the definition of projective degrees below. Consider a rational map $\phi : \pp^n \dashrightarrow \pp^m$. In the manner of  Harris (Example 19.4 of \cite{harris1992algebraic}) we may define the \textit{projective degrees} of the rational map $\phi$ as a list of integers $(g_0, \dots , g_n)$ where \begin{equation}
g_i=\mathrm{card} \left(\phi^{-1} \left( \pp^{ m-i} \right) \cap \pp^{i} \right). \label{eq:projective_degrees_def}
\end{equation}
 Here $\pp^{ m-i} \subset  \pp^m$ and $\pp^{i} \subset \pp^n$ are general hyperplanes of dimension $m-i$ and $i$ respectively and $\mathrm{card}$ is the cardinality of a zero dimensional set. Note that points in $\left(\phi^{-1} \left( \pp^{ m-i} \right) \cap \pp^{i} \right)$ will have multiplicity one (this follows from the Bertini theorem of Sommese and Wampler \cite[\S A.8.7]{sommese2005numerical}).

To compute the projective degrees $g_i$ we may apply Theorem \ref{theorem:projective_deg_theorem} below. This computation is probabilistic and yields the correct result for a choice of objects lying in an open dense Zariski set of the corresponding parameter space.

\begin{theorem}[Theorem 4.1 of \cite{helmer2014algorithm}]
Let $I=(f_0,\dots,f_m)$ be a homogeneous ideal in $k[x_0, \dots , x_n]$ defining an $r $-dimensional scheme $V=V(I)$, and assume, without loss of generality that all the polynomials $f_i$ generating $I$ have the same degree.  The projective degrees $(g_0,\dots , g_n)$ of $ \phi : \pp^n  \dashrightarrow \pp^m$,  \begin{equation}
\phi:p \mapsto \left( f_0(p): \cdots : f_m(p) \right), \label{eq:rational_map_of_an_ideal}
\end{equation}are given by \begin{equation}
g_i= \dim_k \left( k[x_0, \dots , x_n,T]/(P_1 +\cdots +P_{i}+L_1+\cdots +L_{n-i}+L_A+S)\right).
\end{equation} Here $P_{\ell},L_{\ell},L_A$ and $S$ are ideals in $k[x_0,\dots , x_n,T]$  with \begin{align*}
P_{\ell}& = \left( \sum_{j=0}^m \lambda_{\ell,j} f_j\right), \;\;\; \lambda_{\ell,j} \mathrm{\; a \; general \; scalar \; in \;}k,\; \ell = 1,\dots,n,\\
S& =  \left( 1-T \cdot \sum_{j=0}^m \vartheta_{j} f_j \right), \;\;\; \vartheta_{j} \mathrm{\; a \; general \; scalar \; in \;}k,\\
L_{\ell}&= \left( \sum_{j=0}^n \mu_{\ell,j} x_j \right), \;\;\; \mu_{\ell,j} \mathrm{\; a \; general \; scalar \; in \;}k,\; \ell = 1,\dots,n,\\
L_{A}&= \left(1- \sum_{j=0}^n \nu_{j} x_j \right), \;\;\; \nu_{j} \mathrm{\; a \; general \; scalar \; in \;}k.\\
\end{align*}\label{theorem:projective_deg_theorem} 
Additionally $g_0=1$.  
\end{theorem}

Finally we give an expression due to Aluffi \cite{aluffi2003computing} for the Segre class of a projective scheme in terms of the projective degrees defined above. The expression in (\ref{eq:Segre_gs}) combined with Theorem \ref{theorem:projective_deg_theorem} will allow us to compute Segre classes of projective schemes. For more details see \cite{helmer2014algorithm}. 
\begin{propn}[Proposition 3.1 of \cite{aluffi2003computing}] Let $I=(f_0, \dots , f_m) \subset k[x_0, \dots , x_n] $ be a homogeneous ideal defining a scheme $Y\subset \pp^n$ and let $h=c_1\left( \oo_{\pp^n}(1) \right)$ be the class of a hyperplane in $A^*(\pp^n)$. Since $I$ is homogeneous we may assume that the $\deg (f_i) =d$ for all $i$. Let $\phi: \pp^n \dashrightarrow \pp^m $ be the rational map specified by  $$
p \mapsto (f_0(p):\cdots: f_m(p)),$$ let $(g_0,\dots,g_n)$ be the projective degrees of $\phi$. Then we have: \begin{align}
s(Y,\pp^n) &= 1 - c(\oo(dh))^{-1} \cap \left(\sum_{i=0}^n \frac{g_ih^i}{c(\oo(dh))^i} \right) \\
&={1-\sum_{i=0}^n\frac{g_i h^i}{(1+dh)^{i+1}} \; \in A^*(\pp^n)\cong \ZZ[h]/(h^{n+1}).
\label{eq:Segre_gs} } 
\end{align} \label{prop:SegreProjectiveDegreeAluffi} \end{propn}  
% \section{Projective Degree's Via Mixed Multiplicities}
\section{The Algorithm to Compute the $c_{SM}$ Class of a Projective Complete Intersection} \label{section:algorithm}

In this section we describe our new algorithm to compute the $c_{SM}$ class (and hence the Euler characteristic) of a complete intersection subscheme of $\pp^n$ over an algebraically closed field of characteristic zero.

Let $V=V(f_0,\dots,f_m)$ be a complete intersection subscheme of $\pp^n$ such that the scheme $ V(f_0,\dots,f_{m-1})$ is non-singular (allowing for a possible reordering of the generators) and let $J$ be the ideal generated by the $(m+1)\times (m+1)$ minors of the Jacobian matrix of partial derivatives of $f_0,\dots,f_m$. The primary result needed for the algorithms described below is given in Theorem \ref{theorem:MainTheorem} which gives a formula for $c_{SM}(V)$ in terms of the Segre class of $ s(Y,\pp^n)$ where $Y=V(J)\cap V$ is the singularity subscheme of $V$. This Segre class can then be computed using \eqref{eq:Segre_gs} and a method to compute the projective degrees such as Theorem \ref{theorem:projective_deg_theorem}. Theorem \ref{theorem:MainTheorem} follows from Theorem 1.1 of Fullwood \cite{fullwood2014milnor}. We summarize this method in Algorithm \ref{algorithm:mainAlg}. 

In Proposition \ref{propn:Inclusion_exclusion_singular_part_only} and Corollary \ref{corr:inculusionexclusion_ignore_smooth_big} we extend the result of Theorem \ref{theorem:MainTheorem} to any (global) complete intersection subscheme of $\pp^n$ with a type of inclusion-exclusion which considers only the singular generators of the ideal. Hence the number of required Segre class computations is exponential in the number of singular generators. At worst, if all generators define singular schemes, this reduces to inclusion-exclusion as in Proposition \ref{propn:csm_higher_codim}. We present this generalized version of Algorithm \ref{algorithm:mainAlg} in Algorithm \ref{algorithm:hybrid} below.  

In Section \ref{subsection:Runing_time_comparison} we compare the running time of Algorithm \ref{algorithm:hybrid} described below to other algorithms to compute $c_{SM}$ classes for complete intersection varieties in $\pp^n$. We see that for many of the cases considered the new algorithm does indeed provide a performance improvement. While the new method to compute $c_{SM}$ classes is not applicable in all cases it does seem to complement existing methods by providing an efficient approach for a certain subset of problems, particularly those where the ideal defining a complete intersection $V$ has only a few generators which define a singular scheme.  

In \S\ref{subsection:probabilityTheoretic} we obtain a theoretical bound on the probability that Algorithm \ref{algorithm:mainAlg} (or equivalently Algorithms 1 and 2 of the author \cite{helmer2014algorithm}) will give a correct result when working over a finite parameter space. In \S\ref{subsection:probExper} we consider the probability of success for Algorithm \ref{algorithm:mainAlg} (or equivalently Algorithms 1 and 2 of \cite{helmer2014algorithm}) when preformed on several examples using our test computation environment. 
\subsection{The Main Result}
 Let $M$ be a smooth algebraic variety and let $V$ be a subscheme of $M$. From the definition of the Milnor class in \eqref{eq:milnor_class_def} we have the following formula for the class $c_{SM}(V)$ in $A^*(M)$:{\begin{equation}
c_{SM}(V)=c_{FJ}(V)-(-1)^{\codim(V)} \mathcal{M}(V). \label{eq:CsmCfjMilnorRelation}
\end{equation}}%Here $c_{FJ}(X)$ is the Chern-Fulton-Johnson class and $c_{SM}(X)$ is the Chern-Schwartz-MacPherson class. We will consider all the above characteristic classes as elements of the the Chow ring of $M$, $A^*(M)$.
 We now define several notations of Aluffi \cite[\S 1.4]{aluffi1995singular} for operations in the Chow ring. Let $\alpha = \sum_{i\geq 0}\alpha^{(i)}$ be a cycle class in $A^*(M)$ with $\alpha^{(i)}$ denoting the piece of $\alpha$ of codimension $i$ in $A^*(M)$, that is $\alpha^{(i)} \in A^i(M)$. Also let $\mathcal{L}$ be some line bundle on $M$. Define the following notations,
 \begin{equation}
  \alpha^{\vee} = \sum_{i \geq 0} (-1)^i \alpha^{(i)}, \mathrm{\; \;\; and \; \;\;} \alpha \otimes_M \mathcal{L} = \sum_{i \geq 0} \frac{\alpha^{(i)}}{c(\mathcal{L})^i}. \label{eq:ALuffi_chow_ring_tensor_notations}
\end{equation} 
 %$$ \alpha \otimes_M \mathcal{L} = \sum_{i \geq 0} \frac{\alpha^{(i)}}{c(\mathcal{L})^i}. $$
 %One may check that the notation $\alpha \otimes_M \mathcal{L}$ yields a well defined action on $A^*(M)$, see Aluffi \cite[\S 2.3]{aluffi2005characteristic}. 

In \cite[\S 1.1]{fullwood2014milnor}, Fullwood gives a new formula for the Milnor class of a subscheme $V\subset M$ which is a global complete intersection of any codimension with an additional assumption on the structure of $V$. %We state the result of Fullwood \cite{fullwood2014milnor} below.

\begin{theorem}[Theorem 1.1 of Fullwood \cite{fullwood2014milnor}] Let $M$ be a smooth algebraic variety over an algebraically closed field of characteristic zero. Let $V$ be a possibly singular global complete intersection corresponding to the zero scheme of a vector bundle $\mathcal{E} \to M$. Let $j=\mathrm{rk}(\mathcal{E})$. Additionally assume that $ V=M_1\cap \cdots \cap M_j$ for some hypersurfaces $M_1, \dots , M_j$ and assume that, for some ordering of the hypersurfaces, $M_1\cap \cdots \cap  M_{j-1}$ is smooth. Let $\mathcal{L}\to M$ denote the line bundle associated to the divisor $ M_j$ and let $Y$ denote the singularity subscheme of $V$. Then we have \begin{equation}
\mathcal{M}(V)=\frac{c(T_M)}{c(\mathcal{E}) } \cap \left( c( \mathcal{E}^{\vee} \otimes \mathcal{L}) \cap \left( s(Y,M)^{\vee } \otimes_M \mathcal{L} \right) \right).  \label{eq:MilnorOneSingGen}
\end{equation}
\end{theorem}
Note that if $V$ is non-singular we will have that $\mathcal{M}(V)=0$. 

\begin{remark}
We also note that if $V=V(I)$ is a non-singular subscheme of $\pp^n$ (even if it is not a complete intersection) we may simply write the following in $A^*(\pp^n)\cong \ZZ[h]/(h^{n+1})$: \begin{equation}
c_{SM}(V)=c(T_{\pp^n})\cap s(V,\pp^n)=(1+h)^{n+1} s(V,\pp^n).
\end{equation} Hence we need compute only the Segre class $s(V,\pp^n)$; this can be done directly by calculating the projective degrees of the rational map specified by the ideal $I$ using Theorem \ref{theorem:projective_deg_theorem} and then applying the result of Proposition \ref{prop:SegreProjectiveDegreeAluffi} to obtain the Segre class. Thus, in particular, inclusion-exclusion is not required in the smooth case. See Fulton \cite[\S4.2.6]{fulton} or Aluffi \cite{aluffi2005characteristic} for more details. \label{remark:SmoothCSM}
\end{remark}

Combining the relation (\ref{eq:CsmCfjMilnorRelation}), the result of Fullwood \cite{fullwood2014milnor} given in \eqref{eq:MilnorOneSingGen}, and the expression for the $c_{FJ}$ class of a locally complete intersection of Suwa \cite{suwa1997classes} we obtain Theorem \ref{theorem:MainTheorem}. This result combined with Proposition \ref{propn:Inclusion_exclusion_singular_part_only} will allow us to devise a more efficient algorithm to compute $c_{SM}$ classes of possibly singular complete intersection varieties. % giving an expression for $c_{SM}(V)$, assuming that $V$ can be written as the intersection of $j$ hypersurfaces such that some intersection of $j-1$ of the hypersurfaces is smooth (scheme theoretically). Specifically, Fullwood \cite{fullwood2014milnor} gives the following result:

%This gives a method to compute $c_{SM}$ classes of higher codimensional complete intersection schemes having the structure described above. We will express this relation in the explicit form used by our implementation of the algorithm to compute $c_{SM}$ classes for complete intersections.  
\begin{theorem}
Let $k$ be an algebraically closed field of chacteristic zero and let $I=(f_0,\dots, f_m)$ be a homogeneous ideal in $k[x_0,\dots,x_n]$. Assume that $V=V(I)$ is a complete intersection subscheme of $\pp^n$ and let $Y$ be the singularity subscheme of $V$. Let $\deg(f_i)=d_i$, and further assume that $V(f_0,\dots,f_{m-1})$ is smooth scheme theoretically. Let $$A^*(\pp^n)\cong \ZZ[h]/(h^{n+1})$$ denote the Chow ring of $\pp^n$ where $h=c_1(\oo_{\pp^n}(1))$ is the hyperplane class in $\pp^n$. Then we have the following relation in $A^*(\pp^n)$ :  { \begin{align*}
&c_{SM}(V )= (1+h)^{n+1} \cdot \prod_{i=0}^{m} \frac{d_ih}{1+d_ih}- \\
&\frac{(-1)^m(1+h)^{n+1}}{\prod_{i=0}^m (1+d_i h)} \left( \sum_{p=0}^{m} h^p \sum_{i=0}^p {m -i \choose p-i} (-1)^i d_m^{p-i}\cdot \tilde{c}_i  \right) \cdot \left( \sum_{i=0}^n \frac{(-1)^i s_i h^i}{(1+d_m)^i}\right),
\end{align*}
}\normalsize  where we write $$\prod_{i=0}^m (1+d_i h)=\sum_{i=0}^m \tilde{c}_i h^i, \;\;\; \mathrm{and} \;\;\; s(Y,\pp^n)=\sum_{i=0}^n s_i h^i.$$ \normalsize \label{theorem:MainTheorem}
\end{theorem} 
\begin{proof}
 First consider the result of \eqref{eq:MilnorOneSingGen}, taking $M=\pp^n$. Since $V$ is a complete intersection it may be defined as the zero scheme of a rank $m+1$ vector bundle $\mathcal{E}$. Let $\mathcal{L} \to \pp^n $ be the line bundle associated to $V(f_m)$. Then we have that $\mathcal{L}=\oo(d_m h)$, $c(\mathcal{E})=\prod_{i=0}^m (1+d_i h)$ and $c(T_{\pp^n})=(1+h)^{n+1}$. Combining this with \eqref{eq:MilnorOneSingGen} we have \begin{align*}
\mathcal{M}(V ) &=\frac{c(T_{\pp^n})}{c(\mathcal{E}) } \cap \left( c( \mathcal{E}^{\vee} \otimes \mathcal{L}) \cap \left( s(Y,{\pp^n})^{\vee } \otimes_{\pp^n} \mathcal{L} \right) \right) \\
&= \frac{(1+h)^{n+1}}{\prod_{i=0}^m (1+d_i h)} \sum_{p=0}^{m} \sum_{i=0}^p {m -i \choose p-i} c_i(\mathcal{E}^{\vee}) c_1(L)^{p-i} \cap \left( s(Y,{\pp^n})^{\vee } \otimes_{\pp^n} \oo(d_mh) \right)
\end{align*}

Let $$c(\mathcal{E})= \prod_{i=0}^m (1+d_i h)=\sum_{i=0}^m \tilde{c}_i h^i, \; \mathrm{and}\; s(Y,\pp^n)=\sum_{i=0}^n s_i h^i,$$% note that $s(Y,\pp^n)$ can be written in terms of projective degrees as in (\ref{eq:Segre_gs}). 
using \eqref{eq:ALuffi_chow_ring_tensor_notations} we may expand the expression $\left( s(Y,{\pp^n})^{\vee } \otimes_{\pp^n} \oo(d_mh) \right)$ as, \begin{align*}
\left( \sum_{i=0}^n s_i h^i  \right)^{\vee } \otimes_{\pp^n} \oo(d_mh) &= \left(\sum_{i=0}^n (-1)^is_i h^i  \right) \otimes_{\pp^n} \oo(d_mh) \\
&=\sum_{i=0}^n \frac{(-1)^is_i h^i}{c \left( \oo(d_mh) \right)^i } \\
&= \sum_{i=0}^n \frac{(-1)^is_i h^i}{ \left(1 +d_mh\right)^i } .
\end{align*} We may now write, \small{ $$
\mathcal{M}(V )=\frac{(1+h)^{n+1}}{\prod_{i=0}^m (1+d_i h)} \left( \sum_{p=0}^{m} h^p \sum_{i=0}^p {m -i \choose p-i} (-1)^id_m^{p-i}\cdot \tilde{c}_i  \right) \cdot \left( \sum_{i=0}^n \frac{(-1)^i s_i h^i}{(1+d_m)^i}\right).
$$} \normalsize

Since $V$ is a complete intersection in $\pp^n$ from Suwa \cite{suwa1997classes} we have $$
c_{FJ}(V)=(1+h)^{n+1} \cdot \prod_{i=0}^{m} \frac{d_ih}{1+d_ih},
$$ and applying the relation $c_{SM}(V)=c_{FJ}(V)-(-1)^{m} \mathcal{M}(V)$ gives the desired result. \end{proof}

Hence we may conclude that the computation of $c_{SM}$ classes in the case of the theorem above requires only the computation of $s(Y,\pp^n)$ (where $Y $ is the singularity subscheme of $V$), which can be accomplished by means of the projective degree calculation of Theorem \ref{theorem:projective_deg_theorem} for the rational map specified by the ideal corresponding to $Y$ and an application of the formula (\ref{eq:Segre_gs}). 

The singularity subscheme $Y$ of $V$ as given above will be $Y=V(J)\cap V$ where $J$ is the ideal in $k[x_0,\dots,x_n]$ generated by the $(m+1) \times (m+1) $ minors of the $(m+1) \times (n+1) $ Jacobian matrix of partial derivatives, i.e.\ the matrix $a_{i,j}=\left(\frac{df_i}{dx_j} \right)$ for $i=0,\dots,m$, $j=0,\dots, n$ (here we index the first row and column of the Jacobian matrix by $0$). In practice we will use the ideal $(I+J):(x_0,\dots,x_n)^{\infty}$ as the ideal of the singularity subscheme $Y$. 

Since the only unknown in the expression of Theorem \ref{theorem:MainTheorem} is the Segre class $s(Y,\pp^n)$ we may obtain an Algorithm to compute $c_{SM}$ classes (in the setting of the theorem) by combining Theorem \ref{theorem:MainTheorem} with the method to compute Segre classes using the projective degree of a rational map given by the author in \cite{helmer2014algorithm}. We summarize this below. 

Let $J=(w_0, \dots , w_m) \subset R=k[x_0, \dots , x_n] $ be a homogeneous ideal defining a scheme $Y\subset \pp^n$ and let $h=c_1\left( \oo_{\pp^n}(1) \right)$ be the class of a hyperplane in $A_*(\pp^n)\cong \ZZ[h]/(h^{n+1})$. Since $J$ is homogeneous we may assume that the $\deg (w_i) =d$ for all $i$. Also let $(g_0,\dots,g_n)$ be the projective degrees of the map $\phi:\pp^n \dashrightarrow \pp^m $, \begin{equation*}
\phi: p \mapsto (w_0(p): \cdots : w_m(p)).
\end{equation*} To compute the projective degrees $g_i$ in the case where $\phi $ is specified by a homogeneous ideal we may apply Theorem \ref{theorem:projective_deg_theorem}. Once we have obtained the projective degrees then we may apply Proposition \ref{prop:SegreProjectiveDegreeAluffi} to obtain the Segre class $s(Y,\pp^n)$. 

To extend the result of Theorem \ref{theorem:MainTheorem} to any complete intersection subscheme of $\pp^n$ we will use Proposition \ref{propn:Inclusion_exclusion_singular_part_only} below. For a scheme $V=V(I) \subset \pp^n$ this proposition describes a type of inclusion-exclusion for $c_{SM}$ class which considers only the generators of $I$ which define singular subschemes. If the majority of generators of $I$ define a non-singular subscheme of $\pp^n$ this result combined with Theorem \ref{theorem:MainTheorem} can sometimes offer a speed advantage in comparison to methods which use only inclusion-exclusion.  
\begin{propn}
Let $Z\subset \pp^n$ be smooth (scheme-theoretically) and let $ X_1=V(f_1),\; X_2=V(f_2)$ be singular hypersurfaces in $\pp^n$. If $V=Z\cap X_1\cap X_2$, then we have \begin{equation}
c_{SM}(V)=c_{SM}(Z\cap X_1) + c_{SM}(Z\cap X_2)-c_{SM}(Z\cap (X_1 \cup X_2)), \label{eq:IE_to_extend_on_sing_gen_CSM}
\end{equation}
 here $X_1 \cup X_2$ is the scheme generated by $f_1 \cdot f_2$. Additionally, when $V$ is a complete intersection each of the terms in (\ref{eq:IE_to_extend_on_sing_gen_CSM}) can be computed using Theorem \ref{theorem:MainTheorem}. \label{propn:Inclusion_exclusion_singular_part_only}
\end{propn} \begin{proof}
This result follows directly from the inclusion-exclusion property of the $c_{SM}$ class, see \eqref{eq:csm_inclusion_exclusion}.
\end{proof}
\begin{corr}
Let $	V = Z \cap V(f_1)\cap  \cdots \cap V(f_r)$ be a subscheme of $\pp^n$, with the subscheme $Z$ being non-singular. Write the polynomials defining $W=V(f_1) \cap \cdots \cap V(f_r)$ as $F=(f_1,\dots , f_r)$ and let $F_{ \left\lbrace S \right\rbrace } = \prod_{i \in S} f_i $ for $S \subset \left\lbrace 1, \dots , r\right\rbrace .$ Then, $$
c_{SM}(Z\cap W)= \sum_{S \subset \left\lbrace 1, \dots , r\right\rbrace} (-1)^{|S|+1}c_{SM} \left(Z\cap V( F_{\left\lbrace S \right\rbrace } )\right)
$$  where $|S|$ denotes the cardinality of the integer set $S$. The expressions \newline $c_{SM} \left(W\cap V( F_{\left\lbrace S \right\rbrace } )\right)$ can be computed using Theorem \ref{theorem:MainTheorem} when $V$ is a complete intersection.\label{corr:inculusionexclusion_ignore_smooth_big}
\end{corr}
 This result allows us to extend the application of Theorem \ref{theorem:MainTheorem} to complete intersections $V=V(I) \subset \pp^n$ where several of the generators of the ideal $I$ define a singular scheme. At worst, when all of the generators are singular, this will reduce to inclusion-exclusion. However if only a few of the generators are singular this could offer a significant computational speed boost by lowering the degrees considerably. 

In Algorithm \ref{algorithm:mainAlg} we summarize the algorithm to compute $c_{SM}$ classes for projective varieties $V$ satisfying the assumptions of Theorem \ref{theorem:MainTheorem}. In Algorithm \ref{algorithm:hybrid} we give an algorithm which is applicable for any subscheme $V$ of $\pp^n$ defined by a homogeneous ideal. This algorithm takes advantage of the result of Corollary \ref{corr:inculusionexclusion_ignore_smooth_big} combined with Theorem \ref{theorem:MainTheorem} when $V$ is a complete intersection. If $V$ is smooth the result of Remark \ref{remark:SmoothCSM} is used. If $V$ is neither smooth nor a complete intersection then inclusion-exclusion is used. 
%The above algorithm can be adapted for use on any complete intersection variety $V=V(I)$ using the inclusion/exclusion property of $c_{SM}$ classes on all generators of $I$ which do not define a smooth scheme as in Proposition \ref{propn:Inclusion_exclusion_singular_part_only}. We summarize an algorithm for this below in Algorithm \ref{algorithm:hybrid}
Below we present Algorithm \ref{algorithm:mainAlg}, a probabilistic algorithm to compute $c_{SM}(V)$ for $V=V(f_0,\dots,f_m)$ where $V(f_0,\dots,f_{m-1})$ is smooth (scheme theoretically).
\begin{algorithm}

%\begin{itemize}

 \textbf{Input:} A homogeneous ideal $I=(f_0,\dots, f_m)$ in $k[x_0,\dots,x_n]$ defining a complete intersection scheme $V=V(I)\subset \pp^n$ such that $V(f_0,\dots,f_{m-1})$ is smooth (scheme theoretically). \newline 
 \textbf{Output:} $c_{SM}(V)$ in $A^*(\pp^n)\cong \ZZ[h]/(h^{n+1})$ and/or $\chi(V)$. 
\begin{itemize}
\item Find the singularity subscheme $Y=V(J)$, of $X$
\begin{itemize}
\item Set $K$ equal to the $(m+1)\times (m+1)$ minors of the Jacobian matrix of $I$, that is the matrix with entries $a_{i,j}=\left(\frac{df_i}{dx_j} \right)$ for $i=0,\dots,m$, $j=0,\dots, n$. 
\item $J=(K+I):(x_0,\dots,x_n)^{\infty}$.
\item $Y=V(J)$.
\end{itemize} 
\item Apply Theorem \ref{theorem:projective_deg_theorem} with the rational map defined by the ideal $J$ to compute the projective degrees $g_0,\dots, g_n$.
\item Compute $s(Y,\pp^n)$ by using (\ref{eq:Segre_gs}) and the projective degrees $g_0,\dots, g_n$ computed above. 
\item Apply Theorem \ref{theorem:MainTheorem} to obtain $c_{SM}(V)$. 
\end{itemize}\label{algorithm:mainAlg}
%\end{itemize}
%\caption[An algorithm to compute the Chern-Schwartz-MacPherson class without using inclusion/exclusion for certain subschemes of $\pp^n$]{An algorithm to compute $c_{sm}(V)$ for $V=V(f_0,\dots,f_m)$ where $V(f_0,\dots,f_{m-1})$ is smooth (scheme theoretically). }
\end{algorithm} 

Below we present Algorithm \ref{algorithm:hybrid}, a probabilistic algorithm to compute $c_{SM}(V)$ for $V=V(I)$ any subscheme of $\pp^n$. This algorithm takes advantage of the result of Corollary \ref{corr:inculusionexclusion_ignore_smooth_big} combined with Theorem \ref{theorem:MainTheorem} when $V$ is a complete intersection. If $V$ is smooth the result of Remark \ref{remark:SmoothCSM} is used.
\begin{algorithm}
%\begin{itemize}
%\footnotesize{
 \textbf{Input:} a homogeneous ideal $I=(f_0,\dots, f_m)$ in $k[x_0,\dots,x_n]$ defining a scheme $V=V(I)\subset \pp^n$. \newline 
\textbf{Output:} $c_{SM}(V)$ in $A^*(\pp^n)\cong \ZZ[h]/(h^{n+1})$ and/or $\chi(V)$. 
\begin{itemize}
\item \textbf{if} $V$ is non-singular (i.e.\  if the singularity subscheme  $Y$ of $V$ is empty):
\begin{itemize}
\item \textbf{if} $\codim(V)=m+1$ (i.e.\ $V$ is a complete intersection):\begin{itemize}
\item $V$ is smooth so $s(Y,\pp^n)=0$ in Theorem \ref{theorem:MainTheorem}, let $d_i=\deg(f_i)$.
\item $c_{SM}(V)=(1+h)^{n+1} \cdot \prod_{i=0}^{m} \frac{d_ih}{1+d_ih} $. 
\item \textbf{Return} $c_{SM}(V)$ and/or $ \chi(V)$.
\end{itemize}
\item Compute the projective degrees $(g_0,\dots,g_n)$ of the rational map defined by the ideal $I$ using Theorem \ref{theorem:projective_deg_theorem}.
\item Compute $s(V,\pp^n)$ by using Eq.\ (\ref{eq:Segre_gs}) and the projective degrees $(g_0,\dots,g_n)$ obtained above. 
\item Compute $c_{SM}(V)=(1+h)^{n+1}s(V,\pp^n)$.
\item \textbf{Return} $c_{SM}(V)$ and/or return $ \chi(V)$.
\end{itemize}
\item \textbf{else if} $\codim(V)=m+1$ (i.e.\ $V$ is a complete intersection):\begin{itemize}
\item \textbf{for} $j=1,..,m$ and \textbf{for} each subset $f_{\ell_0},\dots,f_{\ell_{m-j}}$ of $f_1,\dots, f_m$ containing $m+1-j$ elements:
\begin{itemize}
 \item \textbf{if} $V(f_{\ell_0},\dots,f_{\ell_{m-j}})$ is non-singular:
 \begin{itemize}
 \item Let $Z=V(f_{\ell_0},\dots,f_{\ell_{m-j}})$.
 \item Let $F$ be the set $f_{\ell_{m-j+1}},\dots,f_{\ell_{m}} $ and let $F_{ \left\lbrace S \right\rbrace } = \prod_{i \in S} f_i $ for $S \subset \left\lbrace \ell_{m-j+1}, \dots , \ell_m \right\rbrace$. 
 \item Apply Corollary \ref{corr:inculusionexclusion_ignore_smooth_big} to obtain $$
c_{SM}(V)= \sum_{S \subset \left\lbrace \ell_{m-j+1}, \dots , \ell_m \right\rbrace} (-1)^{|S|+1}c_{SM} \left(Z\cap V( F_{\left\lbrace S \right\rbrace } )\right)
 $$ and compute each $c_{SM}$ class in the summation using Theorem \ref{theorem:MainTheorem} as presented in Algorithm \ref{algorithm:mainAlg}.
%\item Calculate each $c_{SM}$ class above using Theorem \ref{theorem:MainTheorem}. 
%(since each class will be the class of a scheme of the type that satisfiys the hypothesis of Theorem \ref{theorem:MainTheorem},  i.e.\ removing a certain generator will give a smooth scheme). 
\item \textbf{Return} $c_{SM}(V)$ and/or $\chi(V)$.
 \end{itemize}
\end{itemize}
\end{itemize}
\item \textbf{else}: Compute $c_{SM}(V)$  using Algorithm 3 of \cite{helmer2014algorithm}, that is using inclusion/exclusion.% as described in \cite{helmer2014algorithm}; i.e.\ use Proposition \ref{propn:csm_higher_codim}, the expression for the $c_{SM}$ class of a hypersurface in Eq.\ \eqref{eq:csm_hyper} and a method to compute projective degrees such as Theorem \ref{projective_deg_theorem}.
\end{itemize} \label{algorithm:hybrid}
%\end{itemize}
%\caption[A refined algorithm to compute the Chern-Schwartz-MacPherson class of a complete intersection in $\pp^n$ ]{An algorithm to compute $c_{SM}(V)$ for $V=V(I)$ any subscheme of $\pp^n$. This algorithm takes advantage of the result of Corollary \ref{corr:inculusionexclusion_ignore_smooth_big} combined with Theorem \ref{theorem:MainTheorem} when $V$ is a complete intersection. If $V$ is smooth the result of Remark \ref{remark:SmoothCSM} is used. } } \normalsize

\end{algorithm} 

\subsection{Running Time Comparison} \label{subsection:Runing_time_comparison}
\begin{table}[h]\resizebox{\linewidth}{!}{
\begin{tabular}{@{} l *8c @{}}
\toprule 
 \multicolumn{1}{c}{{ \textbf{INPUT}}}    & {CSM (Aluffi \cite{aluffi2003computing}) [M2]}  &  {CSM (Jost \cite{jost2013algorithm}) [M2]}  &  {\textbf{ \color{chameleongreen2} csm\textunderscore dir (Th.\ \ref{theorem:MainTheorem})[M2]} } &  { csm\textunderscore I\textunderscore E (\cite{helmer2014algorithm})[M2]} &  {{  csm\textunderscore dir (Th.\ \ref{theorem:MainTheorem})[Sa]} } &  { csm\textunderscore I\textunderscore E (\cite{helmer2014algorithm})[Sa]} \\ %&  {csm\textunderscore dir (Prop. \ref{propn:projective_deg_via_mixed_mult}) } &  {csm\textunderscore I\textunderscore E (Prop. \ref{propn:projective_deg_via_mixed_mult}) }  \\ 
 \midrule 
   $V_1 \subset \pp^7$ & - & - [-] & \color{chameleongreen2} 0.3s (0.2s) [4.8s] & 580.9s (116.5s) [-]& 0.5s [18.1s] &-[-]\\ % &0.4s & -\\ 
  $V_2 \subset \pp^4$ & -&1.7s [-] & \color{chameleongreen2} 0.3s (0.1s) [1.3s]& 1.2s (1.2s) [44.1s]& 1.4s [1.3s] &6.3s [-]\\ %&3.1s & 4.7s\\
   $V_3 \subset \pp^6$ & -& 27.7s [-] & \color{chameleongreen2} 7.2s (2.2s) [-] & 33.2s (53.2s) [-]& 90.4s [164.3s] & 215.3s[-]\\ %175.4 & -\\
    $V_4 \subset \pp^5$  & -& - [-]& \color{chameleongreen2} 4.6s (0.7s) [5.5s] & - (-) [-]& 31.9s [3.3s] &- [591.6s]\\%&52.3 & -\\
    $V_5 \subset \pp^6$ & -& - [-]& \color{chameleongreen2}  19.9s (7.9s) [24.9s] & - (-) [-]& 145.0s [61.2s] &- [-]\\%& 196.6s & -\\
\bottomrule
 \end{tabular}} \caption{Run times (over $\mathbb{Q}$) of different algorithms for computing $c_{SM}(V)$ and $\chi(V)$ for $V$ a complete intersection subscheme of $\pp^n$. The timings in [\dots] are those of numeric implementations using Bertini \cite{Bertini} or PHCpack \cite{PHCpack}, the numeric timings in the [M2] column use Bertini while the numeric timings in the [Sa] column use PHCpack. The timings in (\dots) are from an implementation of the result of Proposition \ref{theorem:projective_deg_theorem} which uses a saturation rather than computing the degree of the zero dimensional ideal to find the projective degree. For Algorithm 3 of the author \cite{helmer2014algorithm} (denoted csm\textunderscore I\textunderscore E) and for csm\textunderscore dir (Th.\ \ref{theorem:MainTheorem}) we include timings for the Macaulay2 \cite{M2} and Sage \cite{sage} implementations; these are denoted [M2] and [Sa], respectively. %For the methods csm\textunderscore dir (Th. \ref{theorem:MainTheorem} ) and csm\textunderscore I\textunderscore E (Th. \ref{theorem:projective_deg_theorem})
 \label{table:csmResults}}
\end{table}The main computational cost of Algorithm \ref{algorithm:mainAlg} is the computation of the projective degrees $g_0,\dots, g_n$. This can be accomplished in a number of different ways. The method we will use for this computation consists of finding the degree of the zero dimensional ideal described in Theorem \ref{theorem:projective_deg_theorem}. This can be accomplished symbolically using Gr\"{o}bner bases calculations, or numerically using homotopy continuation via a package such as PHCpack \cite{PHCpack} or Bertini \cite{Bertini}. 

For the examples considered the symbolic methods run over a finite field are in general much faster. If, however, one needs to work over $\mathbb{Q}$ the situation is more complicated; for the majority of the examples considered the symbolic methods are still faster but it is less clear that this is always the case for any example. We note, specifically, that when working over $\mathbb{Q}$ if we consider examples with rational coefficients with very large numerators and denominators the numeric methods performance changes very little while the symbolic methods get progressively slower as the size of the integers in the numerators and denominators increases. We also note that the numeric implementations are not necessarily completely optimized and hence with further work on optimizing the implementations their performance could potentially improve, additionally the numeric versions may be run in parallel when required. 

In any case, regardless of the size of the numerator and denominator the relative speeds of the different methods of computing the $c_{SM}$ class do not seem to change in our experience, and hence we give examples with integer coefficients for comparison and because in practice we would prefer to work over a finite field whenever possible.  %One may also use the method of Proposition \ref{propn:projective_deg_via_mixed_mult} described above, however this does not appear to be advantageous in practice. Even though the result of Proposition \ref{propn:projective_deg_via_mixed_mult} appears to preform poorly in comparison with the method of Theorem \ref{theorem:projective_deg_theorem} we still feel that the result is worth note since it does provide an alternative approach, and may offer some sort of benefit for other implementations in the future.   %The computations over local rings seem to be slower and the necessity of preforming a saturation in the algorithm also makes the computation more difficult. However we feel that the result of Proposition \ref{propn:projective_deg_via_mixed_mult} is still worth note since it does provide an alternative approach, and if the speed of local ring computations improves over time this could prove useful. 
  
In Table \ref{table:csmResults} and Table \ref{table:csmResultsFiniteFeild} we give the running times of the algorithm discussed here in comparison to several other algorithms which use inclusion-exclusion to compute the $c_{SM}$ class and Euler characteristic. All methods shown in the tables are implemented and run in Version 1.7 of Macaulay2 \cite{M2} unless otherwise noted, the numeric implementations use Bertini \cite{Bertini} (in the case of the Macaulay2 \cite{M2} version) and PHCpack \cite{PHCpack} (in the case of Sage \cite{sage} versions). Timings for a Sage \cite{sage} implementation of Algorithm \ref{algorithm:hybrid} (denoted csm\textunderscore dir (Th.\ \ref{theorem:MainTheorem}) in Tables \ref{table:csmResults} and \ref{table:csmResultsFiniteFeild}) and Algorithm 3 of \cite{helmer2014algorithm} (that is the inclusion-exclusion only algorithm, denote csm\_I\_E in Tables \ref{table:csmResults} and \ref{table:csmResultsFiniteFeild}) are also given. All test computations were performed on a computer with an Intel i5-450M processor and 4 GB of RAM.

In the tables in this section we take  
\begin{align*}
&V_1=V \left( {21} x_{0}^{2} + {5} x_{1}^{2} - {24} x_{2}^{2}+ 13 x_{3}^{2} + 8 x_{4}^{2} - {10}{6} x_{5}^{2} + 2 x_{6}^{2} +{14} x_{7}^{2},x_1^2x_5-x_0^2x_4 \right), \\
&V_2=V \left( 3x_0^2+19x_1^2+{8}x_2^2+12x_3^2+13x_4^2,{3}{4}x_0+{5}x_1+{19}x_2+{127}x_3-15x_4,\right. \\ &\left.{27}x_0^2-x_4^2 \right), \\
&V_3=V\left(3x_0^2+19x_1^2+{8}x_2^2+12x_3^2+9x_4^2+{3}x_5^2+{2}{5}x_6^2,x_2^3x_3-x_3x_5^3\right), \\
&V_4=V\left(5x_0^2+9x_1^2+{7}{9}x_2^2+2x_3^2+{3}{5}x_4^2+{7}{3}x_5^2,23x_0+9x_1+7x_2+2x_3+4x_4 \right.\\
&+\left. {3}{2}x_5,x_2x_0x_3-x_3x_5x_4 \right), \\
&V_5=V \left({3} x_{0}^{2} + {17} x_{1}^{2} - 47 x_{2}^{2}+ 3 x_{3}^{2} + 38 x_{4}^{2} - {727} x_{5}^{2} + 12 x_{6}^{2},x_0x_6-x_0^2,{43} {x}_{0}^{2}+ \right.\\ 
&{52} {x}_{0} {x}_{1}+94 {x}_{1}^{2}+{5} {x}_{0} {x}_{2}+{13} {x}_{1} {x}_{2}+{x}_{2}^{2}+{x}_{0} {x}_{3}+4 {x}_{1}
     {x}_{3}+{98} {x}_{2} {x}_{3}+{x}_{3}^{2}+{x}_{0} {x}_{4}+ \\
     &{74} {x}_{1} {x}_{4}+{13} {x}_{2} {x}_{4}+{71} {x}_{3} {x}_{4}+{23} {x}_{4}^{2}+{12}
     {x}_{0} {x}_{5} +{2}{x}_{1} {x}_{5}+{x}_{2} {x}_{5}+{65} {x}_{3} {x}_{5}+{92} {x}_{4} {x}_{5}+\\ &\left.{27} {x}_{5}^{2}+{5} {x}_{0} {x}_{6}+{103} {x}_{1}
     {x}_{6}+38 {x}_{2} {x}_{6}+{x}_{3} {x}_{6}+{6} {x}_{4} {x}_{6}+2 {x}_{5} {x}_{6}+{95} {x}_{6}^{2} \right).
\end{align*} 
$V_6$ is a smooth variety of degree eight and codimension three in $\pp^{10}$ defined by three random quadratic forms. $V_7$ is a variety of degree eight and codimension three in $\pp^{10}$ defined by two random quadratic forms and one random degree two polynomial which defines a singular scheme.
\begin{align*}
&V_8=V(-2x_0^3+24x_1^3+x_2^3+x_3^3-7x_4^3,-9x_0^2+43x_1^2+x_2^2-98x_3^2-73x_4^2, x_1x_4 \\ &- x_0x_4, x_1x_0) \\
&V_9=V(-3x_0^3+4x_1^3+x_2^3+x_3^3-7x_4^3-15x_5^3, -31x_0+14x_1-9x_2+17x_3 \\ &-7x_4-15x_5,(x_1 - x_5)x_4,x_3x_0).
\end{align*}For $V_1 \subset\pp^7$ we have $\deg(V_1)=4$ and $\codim(V_1)= 2$, for  $V_2 \subset \pp^4$ we have $\deg(V_2)=4$ and $\codim(V_2)= 3$, for $V_3 \subset \pp^6$ we have $\deg(V_3)=6$ and $\codim(V_3)= 2$, for $V_4 \subset \pp^5$ we have $\deg(V_4)=2$ and $\codim(V_4)= 3$, and for $V_5 \subset \pp^6$ we have $\deg(V_5)=8$ and $\codim(V_5)= 3$. The variety $V_8$ has dimension zero in $\pp^4$ and $\deg(V_8)=24$. The variety $V_9$ has dimension one in $\pp^5$ and $\deg(V_9)=12$.

%In Table \ref{table:csmResultsFiniteFeild} $V_6$ is a smooth variety of degree eight and codimension three in $\pp^{10}$ defined by three random quadratic forms.  

The method CSM (Aluffi \cite{aluffi2003computing}) is the implementation of Aluffi described in \cite{aluffi2003computing}, this implementation uses inclusion-exclusion and considers the projective degrees as the multi-degree of the blowup of $\pp^n$ along the subscheme defined by the partial derivatives for each hypersurface considered in the inclusion-exclusion. This method is implemented in Macaulay2 \cite{M2}. The method CSM (Jost \cite{jost2013algorithm}) is the algorithm of Jost which computes the projective degrees by finding the degrees of residual sets via saturation, this method also uses inclusion-exclusion. This method is implemented in Macaulay2 \cite{M2}. The method csm\textunderscore dir (Th.\ \ref{theorem:MainTheorem}) is the method of Algorithm \ref{algorithm:hybrid}. This method is implemented in Macaulay2 \cite{M2} and in Sage \cite{sage}. The method csm\textunderscore I\textunderscore E (\cite{helmer2014algorithm}) is the method described by the author in \cite{helmer2014algorithm}, this method uses inclusion-exclusion combined with (\ref{eq:csm_hyper}) and uses the result of Theorem \ref{theorem:projective_deg_theorem} to compute the projective degrees. This method is implemented in Macaulay2 \cite{M2} and in Sage \cite{sage}.

In Table \ref{table:csmResults} computations are performed over $\mathbb{Q}$. In Table \ref{table:csmResultsFiniteFeild} computations are performed over $\mathbb{GF}(32749)$. While the $c_{SM}$ class is only defined over fields of characteristic zero doing the computations over  $\mathbb{GF}(32749)$ yields the same $c_{SM}$ classes found by working over $\mathbb{Q}$ for all examples considered here. Previous papers on computing $c_{SM}$ classes such as Aluffi \cite{aluffi2003computing}, Jost \cite{jost2013algorithm} and the author \cite{helmer2014algorithm} have also performed test computations over a finite field.  %The method csm\textunderscore dir (Prop. \ref{propn:projective_deg_via_mixed_mult}) also uses Theorem \ref{theorem:MainTheorem}, but computes the projective degrees using mixed multiplicities via Proposition \ref{propn:projective_deg_via_mixed_mult}. The method csm\textunderscore I\textunderscore E (Prop. \ref{propn:projective_deg_via_mixed_mult}) uses inclusion-exclusion and also uses mixed multiplicities to compute the projective degree.    

%$ $,  $$ $$ 
% $$V_5=V(\frac{3}{2} x_{0}^{2} + \frac{17}{9} x_{1}^{2} - \frac{4}{7} x_{2}^{2}+ 3 x_{3}^{2} + 38 x_{4}^{2} - \frac{72}{7} x_{5}^{2} + 12 x_{6}^{2},x_0x_6-x_0^2,\frac{4}{3} {x}_{0}^{2}+\frac{5}{2} {x}_{0} {x}_{1}+\frac{9}{4} {x}_{1}^{2}+\frac{1}{5} {x}_{0} {x}_{2}+\frac{1}{3} {x}_{1} {x}_{2}+{x}_{2}^{2}+{x}_{0} {x}_{3}+4 {x}_{1}
%     {x}_{3}+\frac{9}{8} {x}_{2} {x}_{3}+{x}_{3}^{2}+{x}_{0} {x}_{4}+\frac{7}{4} {x}_{1} {x}_{4}+\frac{1}{3} {x}_{2} {x}_{4}+\frac{7}{3} {x}_{3} {x}_{4}+\frac{2}{3} {x}_{4}^{2}+\frac{1}{2}
 %    {x}_{0} {x}_{5}+\frac{1}{2}{x}_{1} {x}_{5}+{x}_{2} {x}_{5}+\frac{9}{8} {x}_{3} {x}_{5}+\frac{9}{2} {x}_{4} {x}_{5}+\frac{1}{2} {x}_{5}^{2}+\frac{3}{5} {x}_{0} {x}_{6}+\frac{10}{3} {x}_{1}
  %   {x}_{6}+\frac{3}{8} {x}_{2} {x}_{6}+{x}_{3} {x}_{6}+\frac{1}{6} {x}_{4} {x}_{6}+2 {x}_{5} {x}_{6}+\frac{9}{5} {x}_{6}^{2}).$$ 

For the smooth variety $V_6$ the computation of $c_{SM}(V_6)$ by Algorithm \ref{algorithm:mainAlg} or Algorithm \ref{algorithm:hybrid} calculates the singularity subscheme $Y$ of $V_6$ first, but since $V_6$ is smooth then $s(Y,\pp^n)=0$ is obtained immediately after $Y$ is computed without the need to calculate the projective degrees. Hence in this case very nearly all of the time is spent computing the singularity subscheme $Y$. Similarly, for the variety $V_7$ the computation of $c_{SM}(V_7)$ using Algorithm \ref{algorithm:hybrid} spends the majority of the computation time finding the singularity subscheme of $V_2$ (approximatively $90\%$ of the $59.5$s average runtime). 

For the varieties $V_8$ and $V_9$ the result of Theorem \ref{theorem:MainTheorem} is not directly applicable and hence the method csm\textunderscore dir (Th.\ \ref{theorem:MainTheorem}), which is our implementation of Algorithm \ref{algorithm:hybrid}, must apply Corollary \ref{corr:inculusionexclusion_ignore_smooth_big}. We see that for the case of the variety $V_8$, Algorithm \ref{algorithm:hybrid} still provides a marked advantage in comparison to inclusion-exclusion only. However for $V_9$, Algorithm \ref{algorithm:hybrid} performs worse than the inclusion-exclusion method of csm\textunderscore I\textunderscore E (\cite{helmer2014algorithm}). It seems that in the cases where result of Theorem \ref{theorem:MainTheorem} is not directly applicable the benefit of using Algorithm \ref{algorithm:hybrid} as opposed to the inclusion exclusion algorithm \cite{helmer2014algorithm} is less clear, in some cases Algorithm \ref{algorithm:hybrid} is beneficial and in some cases it is not. The timings of Algorithm \ref{algorithm:hybrid} in the cases where Theorem \ref{theorem:MainTheorem} is not directly applicable are likely strongly dependent on the structure of the singularity subschemes encountered in the partial inclusion exclusion, which is not easy to determine prior to computing these singularity subschemes. It is however not clear to us, at present, precisely what properties of example $V_8$ makes it more advantageous than example $V_9$ for applying Algorithm \ref{algorithm:hybrid} to compute the $c_{SM}$ class; there are many factors at play in such a consideration and it is not clear to us which factor is the primary one.   

\begin{table}[h]\resizebox{\linewidth}{!}{
\begin{tabular}{@{} l *8c @{}}
\toprule 
 \multicolumn{1}{c}{{ \textbf{INPUT}}}    & {CSM (Aluffi) [M2]}  &  {CSM (Jost \cite{jost2013algorithm}) [M2] }  &  {\textbf{ \color{chameleongreen2} csm\textunderscore dir (Th.\ \ref{theorem:MainTheorem})} [M2]} &  { csm\textunderscore I\textunderscore E (\cite{helmer2014algorithm})[M2] } & {{  csm\textunderscore dir (Th.\ \ref{theorem:MainTheorem})} [Sa]} &  { csm\textunderscore I\textunderscore E (\cite{helmer2014algorithm}) [Sa] } \\% &  {csm\textunderscore dir (Prop. \ref{propn:projective_deg_via_mixed_mult}) } &  {csm\textunderscore I\textunderscore E (Prop. \ref{propn:projective_deg_via_mixed_mult}) }  \\ 
 \midrule 
  % $V_1 \subset \pp^7$ & - & 47.6s  & \color{chameleongreen2} .2s &1.7s\\% &0.3s & -\\ 
  %$V_2 \subset \pp^4$ & -&1.7s [-] & \color{chameleongreen2} 0.3s (0.1s) [1.3s]& 1.2s (1.2s) [44.1s]&3.14s & -\\
  % $V_3 \subset \pp^6$ & -& 27.7s [-] & \color{chameleongreen2} 7.2s (2.2s) [-] & 33.2s (53.2s) [-]&- & -\\
  %  $V_4 \subset \pp^5$ & -& -& \color{chameleongreen2} .2s & 1.8s\\%&3.3s & -\\
  %  $V_5 \subset \pp^6$ & -& 132.6& \color{chameleongreen2}  .7s & 8.7s \\%& 24.1s & -\\
  %  $V_6 \subset \pp^{10}$ & -& -& \color{chameleongreen2}  21.5s & -\\%& 21.5s & -\\
   % $V_7 \subset \pp^{10}$ & -& -& \color{chameleongreen2}  59.5s & -\\%& -& -\\
  %  $V_8 \subset \pp^{4}$ & -& 53.1s& \color{chameleongreen2}  8.9s & 30.4s\\%& -& -\\
   % $V_9 \subset \pp^{5}$ & -& 311.5s& \color{chameleongreen2}  76.9s & 78.4s\\%& -& -\\
     $V_1 \subset \pp^7$ & - & 47.6s  & \color{chameleongreen2} 0.2s &1.1s&0.2s &0.6s\\% &0.3s & -\\ 
  $V_2 \subset \pp^4$  & - & 0.3s  & \color{chameleongreen2} 0.1s &0.3s&0.1s &0.4s\\% &0.3s & -\\ 
  $V_3 \subset \pp^6$ & - & 1.5s  & \color{chameleongreen2} 0.2s &0.9s&0.2s &0.4s\\% &0.3s & -\\ 
    $V_4 \subset \pp^5$ & -& -& \color{chameleongreen2} 0.1s & 0.9s&0.2s &1.1s\\%&3.3s & -\\
    $V_5 \subset \pp^6$ & -& 132.6& \color{chameleongreen2}  0.5s & 1.9s&0.5s &7.0s \\%& 24.1s & -\\
    $V_6 \subset \pp^{10}$ & -& -& \color{chameleongreen2}  21.5s &- &3.7s& -\\%& 21.5s & -\\
    $V_7 \subset \pp^{10}$ & -& -& \color{chameleongreen2}  59.5s & 95.8s & 38.7s& 123.2s\\%& -& -\\
    $V_8 \subset \pp^{4}$ & -& 67.9s& \color{chameleongreen2}  0.7s & 20.9s&0.8s &34.1s \\%& -& -\\
    $V_9 \subset \pp^{5}$ & -& 311.5s& \color{chameleongreen2}  19.8s & 5.3s& 13.4s& 11.4s\\%& -& -\\
\bottomrule
 \end{tabular}} \caption{ Run times of different algorithms for computing $c_{SM}(V)$ and $\chi(V)$ for $V $ a complete intersection subscheme of $\pp^n$ . We use - to denote computations that were stopped after ten minutes (600 s). All computations are performed over the finite field $\mathbb{GF}(32749)$. For Algorithm 3 of the author \cite{helmer2014algorithm} (denoted csm\textunderscore I\textunderscore E) and for csm\textunderscore dir (Th.\ \ref{theorem:MainTheorem}) we include timings for the Macaulay2 \cite{M2} and Sage \cite{sage} implementations; these are denoted [M2] and [Sa], respectively.
 \label{table:csmResultsFiniteFeild}}
\end{table}

Overall in Tables \ref{table:csmResults} and \ref{table:csmResultsFiniteFeild} we see that, for the types of examples for which the result of Theorem \ref{theorem:MainTheorem} is applicable it offers a performance increase over the algorithms which use inclusion-exclusion. Additionally we see that overall, at least for our examples, the symbolic implementations tend to be faster than the numeric implementations, even when the symbolic versions run over $\mathbb{Q}$. We also see that we can expect a further speed-up using the symbolic implementations when they are run over a finite field. 

From the results in the tables we can conclude that Algorithm \ref{algorithm:mainAlg} can provide a significant performance improvement for the computation of $c_{SM}(V)$ for some $V$, particularly when $ V=V(f_0,\dots,f_m)$ is a complete intersection subscheme of $\pp^n$ such that $V(f_1,\dots,f_{m-1})$ is smooth. The performance gain offered by Algorithm \ref{algorithm:hybrid} when one must remove several of the generators of $I=(f_0,\dots,f_m)$ to obtain a smooth scheme is less clear, in some cases it seems to offer a performance improvement however in some cases the cost of computing several singularity subschemes and their Segre classes is too great for us to see any benefit in using Algorithm \ref{algorithm:hybrid} over pure inclusion-exclusion. 

In any case Algorithm \ref{algorithm:mainAlg} and Algorithm \ref{algorithm:hybrid} complement other methods to compute $c_{SM}$ classes and Euler characteristics by offering an effective way to improve performance for a certain class of examples. Additionally it seems likely that, with some minor heuristic adjustments to the criterion one uses to decide whether to use the specialized inclusion-exclusion of Corollary \ref{corr:inculusionexclusion_ignore_smooth_big} or the usual inclusion-exclusion of Proposition \ref{propn:csm_higher_codim}, the method of Algorithm \ref{algorithm:hybrid} would be able to offer marked improvement in many cases, and in worst cases to perform similarly to an algorithm using only inclusion-exclusion.

\subsection{Probabilistic Analysis}
\label{subsection:probabilityTheoretic}

Here we consider the probability of correctly computing the $c_{SM}$ class using Algorithm \ref{algorithm:mainAlg} or Algorithm \ref{algorithm:hybrid} above. Note that these algorithms depend on correctly computing the projective degrees $g_i$ using Algorithm 1 of the author \cite{helmer2014algorithm} (which is constructed from the result of Theorem 4.1 of \cite{helmer2014algorithm}, stated as Theorem \ref{theorem:projective_deg_theorem} above), or equivalently on correctly computing the Segre class $s(Y,\pp^n)$ for a certain subscheme $Y$ of $\pp^n$ as described in Algorithm 2 of \cite{helmer2014algorithm}. Recall that the projective degrees $(g_0, \dots, g_n)$ of a rational map $\varphi:\pp^n \dashrightarrow \pp^m$ are defined in \eqref{eq:projective_degrees_def} above; the Segre class of a subscheme of $\pp^n$ is defined in \eqref{eq:segre_def}. 

We note that the probabilistic analysis given in this subsection requires the results proved in Appendix \ref{appendix:explicitRatMap}. We further note that for this subsection, in light of the setting considered in Appendix \ref{appendix:explicitRatMap}, we require only that our field $k$ is algebraically closed and of sufficiently large characteristic and do not require that it has characteristic zero. As mentioned above, due to analytic elements of the construction of the $c_{SM}$ class by MacPherson \cite{macpherson1974chern}, the $c_{SM}$ class is not technically defined in finite characteristic. However the projective degrees and Segre classes are constructed algebraically and are hence defined in this setting. Thus the analysis of this section gives a bound on the probability of success for the computation of the Segre class of a subscheme of $\pp^n$ using Algorithm 2 of \cite{helmer2014algorithm} (which is the only probabilistic computation required for Algorithms \ref{algorithm:mainAlg} and \ref{algorithm:hybrid}) over an algebraically closed field with either characteristic zero or sufficiently large positive characteristic, see Appendix \ref{appendix:explicitRatMap} for further details. 

We note that, speaking in the sense of algebraic geometry in the terminology of books such as Sommese and Wampler \cite{sommese2005numerical}, Algorithm 1, 2 and 3 of \cite{helmer2014algorithm}, and Algorithm \ref{algorithm:mainAlg} or Algorithm \ref{algorithm:hybrid} above yield the correct result for a choice of objects lying in an open dense Zariski set of the corresponding parameter space. In the sense of \S4.4 of Sommese and Wampler \cite{sommese2005numerical} such algorithms are said to succeed with ``algebraic probability one" (this is allowing for an infinite parameter space), see \cite[Definition 4.4.1, Theorem 4.4.2]{sommese2005numerical} for a discussion of this. 

That said, in this subsection, we will instead focus on the probability of computing the correct value for the projective degree $g_i$ using Algorithm 1 of the author \cite{helmer2014algorithm} as it would be implemented on a computer, where we must instead pick random elements from a finite subset of the parameter space. Such a random choice may not necessarily be general in the sense of algebraic geometry. From the probability bound on the correct computation of a given projective degree $g_i$ given in Proposition \ref{propn:prob_analysis} below we immediately obtain a bound on the probability of correctly computing the $c_{SM}$ class using Algorithm \ref{algorithm:mainAlg} or Algorithm \ref{algorithm:hybrid} above since the only probabilistic choices in these algorithms are those involved in computing the projective degrees $(g_0,\dots, g_n)$ of the appropriate rational map. 

\begin{propn}
Suppose we make a random choice of scalars from some finite set $\mathfrak{S}$ in a field $k$ in the steps of Algorithm 1 of \cite{helmer2014algorithm}. Also suppose we have a homogeneous ideal $I=(f_0, \dots, f_m)$ in $k[x_0,\dots,x_n]$ with $\deg(f_j)=d$ for all $j$ (since $V(I)$ in $\pp^n$ we may assume this without loss of generality). Also let $D=(m+n+1)2^n (d+1)^{m+1}$. Then the probability of correctly computing a given projective degree $g_i$ of the rational map $\phi:\pp^n\dashrightarrow \pp^m$, $\phi: p \mapsto (f_0(p):\cdots:f_m(p))$ defined by the ideal $I$ using Algorithm 1 of \cite{helmer2014algorithm} by choosing random scalars in $\mathfrak{S}$ is $$
P(g_i\; \mathrm{is\; computed\; correctly})\geq\left(1-\frac{d^{2i}}{|\mathfrak{S}|} \right)\cdot \left(1-\frac{D}{|\mathfrak{S}|} \right).
$$\label{propn:prob_analysis}
\end{propn}
%Also see Lemma 3.14 of http://mate.dm.uba.ar/~krick/JeKrSaSo03.pdf The computational complexity of the Chow form Gabriela Jeronimo
%Teresa Krick, Juan Sabia and Martin Sombra
\begin{proof}
Algorithm 1 of \cite{helmer2014algorithm} is based on Theorem 4.1 of \cite{helmer2014algorithm}, given as Theorem \ref{theorem:projective_deg_theorem} above. This theorem states that one may compute the projective degrees $(g_0, \dots , g_n)$ as $$
g_i= \dim_k \left( k[x_0, \dots , x_n,T]/(P_1 +\cdots +P_{i}+L_1+\cdots +L_{n-i}+L_A+S)\right).
$$Here $P_{\ell},L_{\ell},L_A$ and $S$ are ideals in $k[x_0,\dots , x_n,T]$  with \begin{align*}
P_{\ell}& = \left( \sum_{j=0}^m \lambda_{\ell,j} f_j\right), \;\;\; \lambda_{\ell,j} \mathrm{\; a \; general \; scalar \; in \;}k,\; \ell = 1,\dots,n,\\
S& =  \left( 1-T \cdot \sum_{j=0}^m \vartheta_{j} f_j \right), \;\;\; \vartheta_{j} \mathrm{\; a \; general \; scalar \; in \;}k,\\
L_{\ell}&= \left( \sum_{j=0}^n \mu_{\ell,j} x_j \right), \;\;\; \mu_{\ell,j} \mathrm{\; a \; general \; scalar \; in \;}k,\; \ell = 1,\dots,n,\\
L_{A}&= \left(1- \sum_{j=0}^n \nu_{j} x_j \right), \;\;\; \nu_{j} \mathrm{\; a \; general \; scalar \; in \;}k.
\end{align*}Let $W=V(P_{1})\cap \cdots \cap V(P_i)\cap V(L_{1})\cap \cdots \cap V(L_{n-i}) \subset k^{i \times m} \times k^{n-i \times n}  \times \pp^n.$ By Proposition \ref{propn:projectiveDegreeExplicit} below we have that the projective degree $g_i$ will be given by $$
g_i = \mathrm{card} (V(P_{1})\cap \cdots \cap V(P_i)\cap V(L_{1})\cap \cdots \cap V(L_{n-i})-V(f_0,\dots,f_m)),
$$ if we choose $\lambda \in k^{i \times m}$ and $\mu \in k^{n-i \times n}$ so that some polynomial $F_1(\lambda, \mu) \neq 0$ where $\deg(F_1)\leq D$. 

Now by the Schwartz-Zippel Lemma (see Schwartz \cite{schwartz1980fast}, Zippel \cite{zippel1979probabilistic}, or DeMillo and Lipton \cite{demillo1978probabilistic}) we have that the probability of $F_1 = 0$ for random choices of $ [\lambda_{\ell,j}]$ and $[\mu_{\ell,j} ]$ in $\mathfrak{S}$ is given by $$
P(F_1 = 0) \leq \frac{D}{|\mathfrak{S}|},
$$ this is the probability that we will randomly choose the scalars $\lambda,\mu$ incorrectly.

%%%%%%%Done Bertini Bit

Now assume that we have chosen our scalars $[\lambda_{\ell,j}]$ and $[\mu_{\ell,j} ]$ correctly, so that the projective degree $g_i$ is the number of points in the zero dimensional set $W=V(P_{1})\cap \cdots \cap V(P_i)\cap V(L_{1})\cdots \cap V(L_{n-i})-V(f_0,\dots,f_m)$ in $\pp^n$. Then for the next step we must choose the scalars $\vartheta_{0},\dots, \vartheta_{m} $ which give the ideal $S$. Suppose that the set $W$ contains points $\left\lbrace p_1, \dots, p_s \right\rbrace$; then we require that our $\vartheta$ be chosen such that $\vartheta_0 f_0(p_j)+ \cdots +\vartheta_m f_m(p_j) \neq 0$ for all $p_j \in W$. Note that it is sufficient that $$
F_2(\vartheta_{0},\dots, \vartheta_{m})=(\vartheta_0f_0(p_1)+ \cdots +\vartheta_m f_m(p_1) )\cdots (\vartheta_0 f_0(p_s)+ \cdots +\vartheta_m f_m(p_s)) \neq 0,
$$ further note that, by B\'ezout's Theorem, $\deg(F_2)=s \leq d^i$. 

Our choice of scalars $\nu_0, \dots,\nu_n$ to define $L_A$ proceeds similarly. Let $y_{\nu}(x_0,\dots,x_n)=\nu_0x_0 + \cdots + \nu_nx_n$, specifically we require that our $\nu$ are chosen such that $y_{\nu}(p_j) \neq 0 $ for all $p_j \in W$. Again it is sufficient that $$
F_3=y_{\nu}(p_1) \cdots y_{\nu}(p_s) \neq 0, 
$$further note that, by B\'ezout's Theorem, $\deg(F_3)=s \leq d^i$. Thus we have that if our $\vartheta$ and $\nu$ are chosen such that $F_1 \cdot F_2 \neq 0$, given the correct choice of our first set of constants, then the projective degree will be computed correctly. 

Now consider a random choice of the scalars $\vartheta$ and $\nu$; we may again apply the Schwartz-Zippel Lemma. Using this Lemma, given that $F_1\neq 0$, we have that the probability of $F_2\cdot F_3 = 0$ for a random choice of our scalars $\vartheta_{0},\dots, \vartheta_{m} $ and $ \nu_0, \dots, \nu_n$ in $\mathfrak{S}$ is $$
P(F_2 \cdot F_3= 0 \; |\;F_1 \neq 0) \leq \frac{d^{2i}}{|\mathfrak{S}|}.
$$This is the probability that we will randomly choose the scalars $\vartheta, \nu$ incorrectly.

Putting this together we have that, for a random choice of all scalars, the probability that we will correctly compute the projective degree $g_i$ is $$P(F_1\cdot F_2 \neq 0 \; \mathrm{AND}\; F_1 \neq 0).$$ We may now write an expression for this: \begin{align*}
P(F_2\cdot F_3 \neq 0 \; \mathrm{AND}\; F_1 \neq 0)&=P(F_2\cdot F_3 \neq 0 \; |\; F_1 \neq 0) \cdot P(F_1 \neq 0) \\
&=(1-P(F_2\cdot F_3 = 0 \; |\; F_1 \neq 0) )\cdot P(F_1 \neq 0)\\
&=(1-P(F_2\cdot F_3 = 0 \; |\; F_1 \neq 0) )\cdot(1- P(F_1 = 0))\\
&\geq\left(1-\frac{d^{2i}}{|\mathfrak{S}|} \right)\cdot \left(1-\frac{D}{|\mathfrak{S}|} \right).
\end{align*}

Thus we have that the probability that we will correctly compute a given projective degree $g_i$ using Algorithm 1 of the author \cite{helmer2014algorithm} by choosing random scalars in a set $\mathfrak{S}$ is $$
P(g_i\; \mathrm{is\; computed\; correctly})\geq\left(1-\frac{d^{2i}}{|\mathfrak{S}|} \right)\cdot \left(1-\frac{D}{|\mathfrak{S}|} \right).
$$Note that as $|\mathfrak{S}|\to \infty$ the probability that $g_i$ is computed correctly goes to one.

\end{proof}

\begin{remark}
Given the result of Proposition \ref{propn:prob_analysis}, the probability that we will correctly compute a given Segre class $s(V(I),\pp^n)$ (for $I=(f_0,\dots,f_m)$ homogenous of degree $d$ in $k[x_0,\dots,x_n]$) using Algorithm 2 of \cite{helmer2014algorithm} by choosing random scalars in $\mathfrak{S} \subset k$ ($|\mathfrak{S}|$ large and finite) is the probability that we will compute all of the associated projective degrees $(g_0,\dots,g_n)$ correctly. Note that by (10) of \cite{helmer2014algorithm}  for $i < \codim(V(I))$ we have that $g_i=d^i$. Also, from \eqref{eq:projective_degrees_def} we have that $g_i=0$ for $i>\min(m,n)$. Now let $D=(m+n+1)2^n (d+1)^{m+1}$ be as in Proposition \ref{propn:prob_analysis}. Then we have that\scriptsize \begin{equation*}
P(s(V(I),\pp^n) \mathrm{\; is \; computed \; correctly }) \geq \left(1-\frac{D}{|\mathfrak{S}|} \right)^{\min(m,n)-\codim(V(I))} \cdot \prod_{i=\codim(V(I))}^{\min(m,n)} \left(1-\frac{d^{2i}}{|\mathfrak{S}|} \right)
\end{equation*}\normalsize  Further note that this directly gives us the probability of success of Algorithm \ref{algorithm:mainAlg} above by choosing random scalars in $\mathfrak{S}$ since Algorithm \ref{algorithm:mainAlg} requires the computation of one Segre class using Algorithm 2 of \cite{helmer2014algorithm}. From this one could also give an estimate on the probability of success of Algorithm \ref{algorithm:hybrid} above for a particular example once the number of required inclusion-exclusion steps was known. % specifically \begin{equation}
%P( \mathrm{Algorithm\; \ref{algorithm:mainAlg} \;fails }) \leq  n\cdot \left(\frac{4(d+1)^{2(n+2)}+2(\max(g_0,\dots,g_n))^2 }{|S|}
%\right).
%\end{equation} The probability of failure for Algorithm \ref{algorithm:hybrid} follows directly from this. Hence we may conclude that Algorithm \ref{algorithm:mainAlg} and Algorithm \ref{algorithm:hybrid} have a high probability of success given that a sufficiently large set $S$ is chosen. 
\end{remark}

\subsection{Experimental Probabilistic Tests}
\label{subsection:probExper}\label{subsection:ExpermentalProbability}
In this subsection we consider several examples to evaluate the probability, in practice, of computing the Segre class correctly using Algorithm 2 of the author \cite{helmer2014algorithm} in our test computation environment using our Macaulay2 \cite{M2} implementation over various finite fields. We note that to compute the $c_{SM}$ class using Algorithm \ref{algorithm:mainAlg} we must compute one Segre class for a subscheme of $\pp^n$; thus this probability is equivalent to the probability that the $c_{SM}$ class of a given example will be computed correctly using Algorithm \ref{algorithm:mainAlg}. When using Algorithm \ref{algorithm:hybrid} we may need to compute several Segre classes, depending on the number of partial inclusion-exclusion steps required.%, hence the probability of correctly computing a single Segre class would also give the corresponding probability that the $c_{SM}$ class of a given example would be computed correctly using Algorithm \ref{algorithm:hybrid} once the number of required partial inclusion-exclusion steps for the given example were known.

This testing is, of course, by no means conclusive but we feel that it does at least provide a rough idea of the minimum sizes of finite fields one should use to have a high probability of obtaining a correct result (when doing these computations with the symbolic implementation over a finite field). Specifically we try running Segre class computations for which we know the correct value over finite fields of different sizes to get an experimental estimate on the probability of failure for the authors algorithm to compute Segre classes (Algorithm 2 of \cite{helmer2014algorithm}); this then corresponds directly to the probability of correctly computing the $c_{SM}$ class using Algorithm \ref{algorithm:mainAlg}. 

Work in $\pp^4=\Proj(k[x_0,\dots, x_4])$, and take $V_1\subset \pp^4$ to be a smooth complete intersection of degree $36$ generated by two polynomials of degree $6.$ By Chapter 4 of Fulton \cite{fulton} (see in particular Example 4.26 and others) we have that $$s(V_1,\pp^4)=48h^4-16h^3+4h^2 \in A^*(\pp^4) \cong \ZZ[h]/(h^5).$$ Now take $V_2 \subset \pp^4$ to be $$V_2=V(x_0^3(x_1x_2^2-x_1^3),x_1x_2^2(x_0^3-x_1x_4x_3),x_4^6-x_4^3x_2^2x_1) \subset \pp^4,$$ using the method of Aluffi \cite{aluffi2003computing} (which is not probabilistic) we may verify that $$
s(V_2,\pp^4)=-1944h^4+54h^3+9h^2 \in A^*(\pp^4) \cong \ZZ[h]/(h^5).
$$ Note that $V_2$ is neither smooth nor a complete intersection; $\deg(V_2)=9$. 

Also take $V_3 \subset \pp^4$ to be $$V_3=V(x_0^6,x_1x_2(x_0^2x_3x_4-x_1x_4^2x_3) ) \subset \pp^4,$$ again using the method of Aluffi \cite{aluffi2003computing} (which is not probabilistic) we may verify that $$
s(V_3,\pp^4)=3888h^4-432h^3+36h^2\in A^*(\pp^4) \cong \ZZ[h]/(h^5).
$$ Note that $V_3$ is not smooth and $\deg(V_3)=36$. 

We have tested our implementation of Algorithm 2 of \cite{helmer2014algorithm} on these examples to see how many incorrect results we can expect in 2000 attempts over finite fields of different sizes and found that, in practice, we have a high probability of obtaining a correct answer with most fields we are likely to want to use.% Note that each time we run Algorithm \ref{algorithm:mainAlg} we must compute one Segre class, hence Table \ref{table:ProbResults} can be thought of as the experimental probability of Algorithm \ref{algorithm:mainAlg} failing to correctly compute the $c_{SM}$ class for a given example in our test computation environment. 

From the experimental results, and from our experience, we feel it is reasonable to conclude that one will obtain the correct result with high probability if one uses a finite field with more than about $25000$ elements, which is quite reasonable when working in a modern computer algebra system. For further assurance one could take a $64$ or $128$ bit prime. We note that for all examples in this subsection we have ran the test Segre class computations below more than $10000$ times working over $\mathbb{GF}(32749)$ and have obtained no incorrect results. We also note that the computation time for our Macaulay2 \cite{M2} implementations for either Algorithm \ref{algorithm:mainAlg}, or Algorithm \ref{algorithm:hybrid} (or of Algorithms 1,2, and 3 of the author \cite{helmer2014algorithm}) changes very little when one moves between the different finite fields considered in these tests.

 \begin{table}[h]\resizebox{\linewidth}{!}{
\begin{tabular}{@{} l *{10}c @{}}
\toprule 
 \multicolumn{1}{c}{{ \textbf{Number of Elements in Finite Field}}}&317 &1021 & 3301 &  5743 & 9001 & 19301& 25073 &27077  &32749   \\ %&  {csm\textunderscore dir (Prop. \ref{propn:projective_deg_via_mixed_mult}) } &  {csm\textunderscore I\textunderscore E (Prop. \ref{propn:projective_deg_via_mixed_mult}) }  \\ 
 \midrule 
   \# of times $s(V_1,\pp^4)$ computed incorrectly in 2000 attempts& 10&1 &1 &0& 0 & 0& 0& 0&0 \\ 
     \# of times $s(V_2,\pp^4)$ computed incorrectly in 2000 attempts& 83&40 &12 &6& 4 &1& 0&0&0 \\ 
       \# of times $s(V_3,\pp^4)$ computed incorrectly in 2000 attempts&  7& 0 & 0 &1 & 0 &0& 0&0&0 \\
\bottomrule
 \end{tabular}} \caption{Number of times an incorrect value for the Segre class of a subscheme of $\pp^n$ was computed in 2000 attempts over different finite fields. For these tests our Macaulay2 \cite{M2} implementation of our algorithm (that is Algorithm 2 of \cite{helmer2014algorithm}) running in Version 1.82 of Macaulay2 \cite{M2} was used. 
 \label{table:ProbResults}}
\end{table}

\subsubsection*{Acknowledgements}This research was partially supported by the Natural Sciences and Engineering Research Council of Canada. The author would also like to thank \'Eric Schost for many helpful discussions throughout the preparation of this note. 

\newpage
%\appendix
\begin{appendices}
\begin{center}
\section{Appendix: An Explicit Construction of the Projective Degrees of a Rational Map} \label{appendix:explicitRatMap}
\textit{Appendix by Martin Helmer and \'Eric Schost}
\end{center}

Throughout this appendix we let $k$ denote
any algebraically closed field. Our goal in this section is to give a
quantitative version of the proof given by B{\"u}rgisser, Cucker, and
Lotz~\cite{burgisser2005counting} of the existence of projective
degrees of a rational map.

%% Note that we restrict our discussion to the base field $k=\C$. This is
%% not a loss for this work as the $c_{SM}$ class, our main object of
%% final interest, is only defined in such a setting,
%% see~\cite{macpherson1974chern}. The only part where we use this
%% assumption is for two claims in Subsection~\ref{ssec:known}; we
%% believe that it would be possible to remove this restriction by
%% working over fields of large enough characteristic.

As a matter of convention, upper case letters denote variables and
lower case ones denote vectors with entries in $k$. 

%%%%%%%%%%%%%%%%%%%%%%%%%%%%%%%%%%%%%%%%%%%%%%%%%%%%%%%%%%%%%%%%%%%%%%%%%%%%%%%%

\subsection{Some known results}\label{ssec:known}

Our proof is modelled on the one given by B{\"u}rgisser, Cucker, and
Lotz in~\cite{burgisser2005counting}, and will rely on the same basic
ingredients: namely a result on the cardinality of regular fibers of a
projection, and an algebraic version of Sard's lemma.

Let $f: V \to W$ be a dominant map of irreducible quasiprojective
varieties over $k$, with $V$ and $W$ of the same dimension. We say
that $x \in V$ is a {\em regular point} of $f$ if $x$ is non-singular
on $V$ and $df_x(T_x V) = T_{f(x)} W$ (which implies that $f(x)$ is
non-singular on $W$), and we say that $y \in W$ is a {\em regular
  value} of $f$ if all $x$ in $f^{-1}(y)$ are regular points of $f$.

The following version of Sard's lemma is very close to ones in the
literature, but does not rely on the characteristic of $k$ being $0$.

\begin{lemma}
  With notation as above, let also $D = [k(V):k(W)]$. If $k$ has
  characteristic $0$ or greater than $D$, then the set of regular
  values of $f$ is contained in a hypersurface of $W$.
\end{lemma}
\begin{proof}[Proof (sketch).]
  We will follow the proof of Theorem~A.4.10
  in Sommese and Wampler~\cite{sommese2005numerical}, which holds in characteristic $0$,
  and highlight the only parts of the proof that use this
  assumption. The proof invokes Proposition~III.10.6
  in Hartshorne~\cite{Hartshorne77}, applied to the restriction of $f$ to
  Zariski-dense subsets $V'$ and $W'$ of respectively $V$ and $W$ that
  are both non-singular (the image of the singular points has
  dimension less than $s=\dim(V)=\dim(W)$, so this is harmless). In
  particular, the induced extension $k(W') \to k(V')$ has degree $D$
  as well.

  The proof of Proposition~III.10.6 in Hartshorne~\cite{Hartshorne77} proceeds as
  follows. Let $V''$ be an irreducible component of the set of all $x$
  in $V'$ such that that $df_x(T_x V')$ has dimension less than $s$
  (this set is closed); we will prove that the Zariski closure $W''$
  of $f(V'')$ has dimension less than $s$.  We proceed by
  contradiction, assuming that $W''$ has full dimension $s$; this
  implies in particular that $\dim(V'') = s$. This in turn implies
  that $[k(V'') : k(W'')]=D$, so our characteristic
  assumption shows that $k(W'') \to k(V'')$ is a separable extension.
  This separability statement is the only one in the proof
  of~\cite[Proposition~III.10.6]{Hartshorne77} that used the
  characteristic $0$ assumption; once we ensure it, the rest of 
  the proof follows directly and proves that $W''$ has dimension
  less than $s$, a contradiction.
\end{proof}

Next, we establish a result on the number of points in fibers of
regular values; see a proof in B{\"u}rgisser, Cucker, and
Lotz~\cite{burgisser2005counting} for 
the case where $k=\mathbb{C}$. In scheme-theoretic terms, this result can in
essence be deduced from known claims such as ``a quasi finite and
proper map is finite'' and ``the fibers of a finite and \'etale map at
closed points have constant cardinality''; we will give here a very
direct proof.

\begin{lemma}\label{lemma:fiber}
  Let $V \subset \pp^n \times k^N$ be an irreducible algebraic set of
  dimension $N$, and let $\pi: \pp^n \times k^N \to k^N$ be the
  projection on the second factor; suppose that the restriction $\pi_V$ of
  $\pi$ to $V$ is dominant. Then all fibers of regular values of this
  restriction have the same cardinality.
\end{lemma}
\begin{proof}
  Without loss of generality, we can suppose that the set $W$ of
  regular values of $\pi_V$ is not empty (otherwise there is nothing
  to do); if we were for instance under the assumptions of the
  previous lemma, this would automatically be the case.  Let then $y$
  and $y'$ be two distinct points in $W$. The Jacobian criterion
  implies that the fibers $\varphi= \pi_V^{-1}(y)$ and $\varphi' =
  \pi_V^{-1}(y')$ are finite; we call $d$ and $d'$ their respective
  cardinalities, so that our goal is to prove that $d=d'$.

  Let $L\subset k^N$ be the affine line joining $y$ to $y'$. We let
  $U$ be a parameter for this line, such that $y$ corresponds to
  $U=0$. Consider the preimage $\pi_V^{-1}(L)$; this may not be a
  curve, so we let $C$ be the union of the irreducible components of
  this set whose image by $\pi_V$ is dominant, or equivalently equal
  to the whole line $L$ (since the projection $\pi_V$ is closed).  In
  particular, every irreducible component of $C$ meets the fiber
  $\varphi$ --- we will use this below.

  Let $I \subset k[U,X_0,\dots,X_n]$ be the defining ideal of $C$, and
  define further the extended ideal $I' = I \cdot
  k(U)[X_0,\dots,X_n]$; by construction, this ideal has dimension zero
  (since almost all fibers of $\pi_V$ above $L$ consists only of
  regular points, hence have dimension zero).

%% @article{Rayner68,
%% author = {Rayner, F. J.},
%% title = {An algebraically closed field},
%% journal = {Glasgow Mathematical Journal},
%% volume = {9},
%% issue = {02},
%% year = {1968},
%% issn = {1469-509X},
%% pages = {146--151},
%% }

  To find the roots of $I'$, let us introduce the ring $\mathbb{L}$ of
  all ``generalized power series'' $F=\sum_{i \in \Gamma} f_i {U}^i$,
  where the index set $\Gamma \subset \mathbb{Q}$ (that depends on
  $F$) is well-ordered and all $f_i$'s are in $k$.  It turns out that
  this ring is a domain and contains an algebraic closure of
  $k(U)$, see Rayner~\cite{Rayner68}; this extends the construction of Puiseux
  series to arbitrary characteristic. Thus, we can write the zero-set
  $S$ of $I'$ as $S=(\gamma_{i,0} : \cdots : \gamma_{i,n})_{1 \le i
    \le D}$, with all $\gamma_{i,j}$ in $\mathbb{L}$. We will now
  prove that $d=D$, and we claim that this is enough to
  conclude. Indeed, we can repeat the construction above $y'$ and
  prove that the cardinality $d'$ of the fiber $\pi_V^{-1}(y')$ is
  equal to $D$ as well.

  We can assume without loss of generality that the fiber $\varphi$
  does not meet the hyperplane $X_0=0$; we write $\tilde \varphi$ for
  the affine set obtained from $\varphi$ through the dehomogenization
  $X_0=1$. Thus, it is enough to prove that $\tilde \varphi$ has
  cardinality $D$.

  Because the exponent support is well-ordered, we can define the {\em
    valuation} of a (non-zero) $F \in \mathbb{L}$ as above as the
  rational $\nu(F)=\min(i \in I,\ f_i \ne 0)$; this extends the
  $u$-adic valuation on $k(U)$. We can then rescale all elements in
  $S$ to ensure that all their entries $(\gamma_{i,0}, \dots,
  \gamma_{i,n})$ have non-negative valuations, with at least one of
  them having valuation zero. This allows us to define the {\em limit}
  $S'$ of these elements of $\pp^n(\mathbb{L})$ as the points in
  $\pp^n=\pp^n(k)$ obtained by taking the coefficient of $U^0$
  entry-wise; by construction, all these points lie in the fiber
  $\varphi$ (we call this process taking the {\em limit} of these
  points). Since all points in $\varphi$ are in the affine set $\{X_0
  \ne 0 \}$, we deduce that the elements in $S'$ can be written as
  $(\gamma_{i,0},\gamma_{i,1},\dots,\gamma_{i,n})_{1 \le i \le D}$,
  with $\gamma_{i,0}$ of valuation $0$ for all $i$.
  
  Let $Y_1,\dots,Y_N$ be coordinates in $k^N$; without loss of
  generality we can assume that $Y_1=U$.  Let then $P_1,\dots,P_s$ be
  polynomials in $k[U,Y_2,\dots,Y_N,X_1,\dots,X_n]$ defining the
  affine variety $\tilde V=V \cap \{X_0 \ne 0\}$; taking generic
  linear combinations of them, we obtain polynomials $Q_1,\dots,Q_n$
  such that $\tilde V=V(Q_1,\dots,Q_n)$ holds in an affine
  neighbourhood $\Omega$ of $\tilde \varphi$ and such that the tangent
  space $T_x \tilde V$ is the nullspace of the Jacobian of
  $Q_1,\dots,Q_n$ for all $x$ in $\tilde V \cap \Omega$. The
  assumption that $y$ is a regular value of $\pi_V$ then means that
  the truncated Jacobian of $Q_1,\dots,Q_n$ with respect to
  $X_1,\dots,X_n$ has full rank at every point of $\tilde \varphi$.

  Restricting $Q_1,\dots,Q_n$ to polynomials $q_1,\dots,q_n$ in
  $k[U,X_1,\dots,X_n]$ by setting $Y_2=\cdots=Y_N=0$, we obtain
  polynomials that define $\pi^{-1}(L)$ within $\Omega$; their
  Jacobian matrix with respect to $X_1,\dots,X_n$ still has full rank
  on $\tilde \varphi$. On the other hand, since we saw that all
  irreducible components of $C$ meet the fiber $\varphi$, we deduce
  the the Zariski closure of $V(q_1,\dots,q_n) \cap \Omega$ is the
  affine curve $C \cap \{X_0 \ne 0\}$. This further implies that the
  zero-set of the ideal $J=\langle q_1,\dots, q_n \rangle \cdot
  k(U)[X_1,\dots,X_n]$ is $\tilde
  S=(\gamma_{i,1}/\gamma_{i,0},\dots,\gamma_{i,n}/\gamma_{i,0})_{1 \le
    i \le D} \subset {\rm frac}(\mathbb{L})^n$. Since all $\gamma_{i,0}$'s 
  have valuation $0$, we can define the limit of such elements as we did 
  for vectors over $\mathbb{L}$.

  Newton iteration applied to the polynomials $q_1,\dots,q_n$ shows
  that for any point $g$ in $\tilde \varphi$, there exists a unique
  vector $\gamma$ in $\tilde S$ and whose limit is $g$; that vector
  $\gamma$ has entries that are actually in $k[[U]] \subset
  \mathbb{L}$ (uniqueness in Newton iteration is usually written for
  power series only, but extends readily in our context). Conversely,
  we saw that the limit of any element in $\tilde S$ is in $\varphi$,
  so that $\varphi$ is in one-to-one correspondence with $\tilde S$,
  as claimed.
\end{proof}

%%%%%%%%%%%%%%%%%%%%%%%%%%%%%%%%%%%%%%%%%%%%%%%%%%%%%%%%%%%%%%%%%%%%%%%%%%%%%%%%

\subsection{An explicit construction of projective degrees}

Let $F_0, \dots, F_m$ be homogeneous of degree $d$ in $k[X_0, \dots,
  X_n]$, not all zero. Following B{\"u}rgisser, Cucker, and
Lotz~\cite{burgisser2005counting}, define
the rational map $\varphi:\pp^n \dashrightarrow \pp^m$ by $x \mapsto
(F_0(x): \cdots : F_m(x))$; it is defined over
$\pp^n-V(F_0,\dots,F_m)$.

Recall the definition of the projective degrees $(g_0,\dots,g_{\min(m,n)})$ of
the map $\varphi$ given in~\eqref{eq:projective_degrees_def}: $g_i$ is
obtained through the construction of a set of the form $\lambda \cap
\varphi^{-1}(\hat \lambda)$, where $\lambda$ and $\hat \lambda$ are
generic linear subspaces of respectively $\pp^n$ and $\pp^m$ of
respective codimensions $n-i$ and $i$, for some $i$ in
$\{0,\dots,\min(m,n)\}$.  The main purpose in this appendix is to make
the genericity requirements on $\lambda$ and $\hat \lambda$ entirely
explicit.

In all that follows, an integer $i$ in $\{0,\dots,\min(m,n)\}$ is
fixed. Although it would be natural to work all along with linear
spaces $\lambda$ and $\hat \lambda$ taken in suitable Grassmannians,
we will instead work in affine spaces, in order to be able to state
degree bounds. Thus, in what follows, we will want to identify a
vector $\ell = (\ell_{1,0}, \dots, \ell_{n-i,n})\in k^{(n-i) (n+1)}$
with $n-i$ linear forms over $\pp^n$ (so that we can view $\lambda$ as
$V(\ell(X))$), and similarly a vector $\hat \ell = (\hat{\ell}_{1,0},
\dots, \hat{\ell}_{i,m}) \in k^{i(m+1)}$ will be identified with $i$
linear forms over $\pp^m$. The corresponding hyperplanes may not
intersect properly, but the restrictions on $\ell$ and $\hat \ell$ we
will make below will ensure it. As per our convention, let $L$ denote
indeterminates $L_{1,0},\dots,L_{n-i,n}$, and similarly for $\hat L$.

Starting from
$F_0,\dots,F_m$ as above, up to reordering the $F_i$'s, we can assume
without loss of generality that $F_0$ is not identically zero. We then
define the following objects:
\begin{itemize}
\item $\Gamma_0 = \{(x,y) \in \pp^n \times \pp^m \ \mid \ x_0 F_0(x)
  \ne 0, y = \varphi(x) \}$,
\item $\Gamma \subset \pp^n \times \pp^m$ is the Zariski closure of $\Gamma_0$,
\item $\Gamma_\Sigma =\Gamma \cap (V(X_0 F_0) \times \pp^m )\subset \Gamma$.
\end{itemize}
We will naturally call $\Gamma$ the {\em graph} of $\varphi$. Remark
that although we could have taken the larger set $\{(x,y) \in \pp^n
\times \pp^m \ \mid \ x \notin V(F_0,\dots,F_m), y = \varphi(x) \}$ as
$\Gamma_0$ (as is done in~\cite{burgisser2005counting}), its Zariski
closure would have coincided with $\Gamma$ as defined above.  
\begin{lemma}\label{lemma:gamma}
  The sets $\Gamma_0$ and $\Gamma_{\Sigma}$ are disjoint,
  $\Gamma=\Gamma_0 \cup \Gamma_{\Sigma}$, $\Gamma$ is irreducible of
  dimension $n$ and $\dim(\Gamma_{\Sigma}) < n$. \label{propn:GammaDisjointUnion}
\end{lemma}
\begin{proof}
  The first two items are straightforward.  From Chapter 7
  of Harris \cite{harris1992algebraic} (see page 77-78 specifically), we get
  that $\Gamma$ is irreducible and birationally equivalent to
  $\pp^n$. Since $\Gamma_\Sigma$ is a proper algebraic subset of it,
  it has dimension less than $n$.
\end{proof}

We will define projective degrees $(g_0,\dots, g_{\min(m,n)})$
formally below. For the moment, we point out that they can then be
understood as integers such that the class of $\Gamma$ in $A^*(\pp^n
\times \pp^m) \cong \ZZ[h_1,h_2]/(h_1^{n+1},h_2^{m+1})$ is given by
$$[\Gamma]=g_0h_2^m+g_1h_1h_2^{m-1}+ \cdots + g_nh_1^nh_2^{m-n}, $$
with $g_i=0$ if $m-i <0$. In other words, $g_i$ will be the number 
of solutions of the system 
$$(x,y) \in \Gamma, \ \ell(x)=0,\ \hat \ell(y)=0,$$ for a generic
choice of $\ell$ and $\hat \ell$ in respectively $ k^{(n-i) (n+1)}$
and $ k^{i (m+1)}$, see Harris~\cite[Section~19.4]{harris1992algebraic}.

The main result in this section is then the following proposition.

\begin{propn}
  Let $i$ be in $\{0,\dots,\min(m,n)\}$ and suppose that $k$ has
  characteristic either $0$ or greater than $d^n$. Then, there exists
  a non-zero polynomial $F \in k[L,\hat L]$ of degree at most
  $(m+n+1)2^n (d+1)^{m+1}$, and a non-negative integer $g_i$, such that
  if $F(\ell,\hat \ell)$ is non-zero, then the system of equations and
  inequations
  \begin{align*}
&\left ( \ell_{j,0} x_0 + \cdots + \ell_{j,n} x_m = 0 \right )_{1 \le j \le n-i},\\    
&    \left (\hat \ell_{j,0} F_0(x) + \cdots + \hat \ell_{j,m} F_m(x) = 0 \right )_{1 \le j \le i},\\
&    x \notin V(F_0,\dots,F_m)
  \end{align*}
    has exactly $g_i$ solutions, and all these solutions satisfy $x_0 \ne 0$.
   \label{propn:projectiveDegreeExplicit}
\end{propn} 

The proof of this proposition occupies the rest of this section.  We
first prove a basic result about incidence varieties involving two
families of linear forms; compare with B{\"u}rgisser, Cucker, and Lotz
\cite[Lemma~A.3.i]{burgisser2005counting}. In the discussion below, we
consider algebraic subsets of the ambient space $\pp^n\times \pp^m
\times k^{(n-i) (n+1)} \times k^{i (m+1)}$; these algebraic sets being
quasi-projective, it makes sense to talk about their irreducibility
properties, or their decomposition into irreducibles.

\begin{lemma}\label{lemma:lambda}
  Let $\Lambda$ be an irreducible algebraic subset of $\pp^n\times
  \pp^m$ and define $\Phi \subset \Lambda \times k^{(n-i) 
    (n+1)} \times k^{i (m+1)}$ as 
  $$ \Phi= \left\lbrace (x, y, \ell, \hat{\ell} )\; \mid
  \;  (x, y) \in \Lambda,\ \ell \in k^{(n-i) 
      (n+1)},\ \hat{\ell} \in k^{i  (m+1)},\ \ell(x)=0,\
    \hat{\ell}(y)=0  \right\rbrace.
  $$ Then $\Phi$ is irreducible of dimension $(n-i)(n+1) +i(m+1)+
  \dim(\Lambda) - n$.
\end{lemma}
\begin{proof}
  Let $(\mathfrak{h}_1,\dots,\mathfrak{h}_c)$ be polynomials defining
  the ideal of $\Lambda$; this ideal is prime by assumption. Then, the
  ideal $I=(\mathfrak{h}_1,\dots, \mathfrak{h}_c, L(X), \hat L(Y) ) $
  defines $\Phi$.  Since all equations are homogeneous in the
  variables $L,\hat L$, we can consider the algebraic set $\Phi'$
  they define in $\Lambda \times \pp^{(n-i)(n+1) +i(m+1) -1}$; we will
  prove that $\Phi'$ is irreducible of dimension $(n-i)(n+1) +i(m+1)+
  \dim(\Lambda) - (n-1)$, which is enough to conclude.

  In this context, we can follow verbatim the proof
  of B{\"u}rgisser, Cucker, and
Lotz~\cite[Lemma~A.3.i]{burgisser2005counting}. Let $\pi_1: \Phi' \to
  \Lambda $ be the projection $(x, y, \ell , \hat{\ell} ) \mapsto
  (x,y)$, and let $(x,y)$ be any point in $\Lambda$. The fiber
  $\pi_1^{-1}(x,y)$ is defined by $n$ linear equations that intersect
  properly, so it is irreducible of dimension $(n-i)(n+1) +i(m+1) +
  (n-1)$; since $\Phi'$ and $\Lambda$ are projective, Theorem~11.14
  in Harris~\cite{harris1992algebraic} shows that $\Phi'$ is irreducible, and
  the theorem on the dimension of fibers shows that it has dimension
  $(n-i)(n+1) +i(m+1) + \dim(\Lambda)+ (n-1)$.
\end{proof}

Let us first construct the incidence variety $\Phi_\Gamma$ associated
to the graph $\Gamma$ by this process; it is irreducible of dimension
$(n-i)(n+1) +i(m+1)$. We will denote by $\Psi_\Gamma$ the image
$\pi_2(\Phi_\Gamma)$, where $\pi_2: \pp^n\times \pp^m \times k^{(n-i)
  (n+1)} \times k^{i (m+1)} \to k^{(n-i) (n+1)} \times k^{i (m+1)} $
is the projection on the right-hand factor $(x,y,\ell,\hat \ell)
\mapsto (\ell,\hat \ell)$. We expect $\Psi_\Gamma$ to have full
dimension, but this does not necessarily have to be the case.

We will also need results on another algebraic set obtained by means
of Lemma~\ref{lemma:lambda}, this time starting from
$\Gamma_\Sigma$. Consistent with the notation above let \small $$
\Phi_\Sigma= \left\lbrace (x, y, \ell, \hat{\ell} )\; \mid \; (x, y)
\in \Gamma_\Sigma,\ \ell \in k^{(n-i) (n+1)},\ \hat{\ell} \in k^{i
  (m+1)},\ \ell(x)=0,\ \hat{\ell}(y)=0 \right\rbrace.$$ \normalsize
The set $\Gamma_\Sigma = V(X_0 F_0)$ is not irreducible; writing its
irreducible components as $W_1,\dots,W_s$, we deduce that
$\Phi_\Sigma$ is the union of all $\Phi_j$, where for $j=1,\dots,s$,
$\Phi_j$ is built using the same process, starting from $W_j$.  Then,
Lemma~\ref{lemma:lambda} shows that each $\Phi_j$ is irreducible
dimension $(n-i)(n+1) +i(m+1)- 1$; the dimension claim also holds for
$\Phi_\Sigma$ itself.

Finally, let $\Psi_{\Sigma}= \pi_2 \left(\Phi_{\Sigma}
\right)$. Because this projection is closed, $\Psi_\Sigma$ is an
algebraic set, and the dimension bound on $\Phi_\Sigma$ implies that
$\Psi_\Sigma$ has codimension at least one. Our first criterion for
choosing $(\ell,\hat \ell)$ will be that they avoid $\Psi_\Sigma$ (as
in~\cite{burgisser2005counting}, the other constraint will amount to
saying that they are regular values of the restriction of $\pi_2$ to
$\Phi_\Gamma$). The following lemma shows a trivial consequence of
such an assumption; in what follows, we use the open set $\Omega =
\{(x,y) \in \pp^n \times \pp^m \ \mid \ x_0 F_0(x)y_0 \ne 0\}$.

\begin{lemma}\label{lemma:avoidSigma}
  Let $(\ell,\hat \ell)$ be in $k^{(n-i) (n+1)} \times k^{i (m+1)} -
  \Psi_\Sigma$. Then, for any point $(x,y,\ell,\hat \ell)$ in
  $\Phi_\Gamma$ lying over $(\ell,\hat \ell)$, $(x,y)$ is in $\Gamma_0
  \cap \Omega$.
\end{lemma}
\begin{proof}
  Let $(\ell,\hat \ell)$ be as above, and let $(x,y,\ell,\hat \ell)$
  be a point in $\Phi_\Gamma$ lying over $(\ell,\hat \ell)$. Because
  $(\ell,\hat \ell)$ is not in the projection of $\Phi_\Sigma$,
  $(x,y)$ is not in $\Gamma_\Sigma$, so by Lemma~\ref{lemma:gamma}, it
  is in $\Gamma_0$; thus we have that $x_0$, $F_0(x)$ are non-zero and that $y =
  \varphi(x)$. This implies that $y_0$ is non-zero as well, so $(x,y)$
  is in the open set~$\Omega$.   
\end{proof}

\begin{corr}\label{coro:oneone}
  Let $(\ell,\hat \ell)$ be in $k^{(n-i) (n+1)} \times k^{i (m+1)} -
  \Psi_\Sigma$. Let $S \subset \pp^n \times \pp^m$ be the fiber
  $\pi_2^{-1}(\ell,\hat \ell) \cap \Phi_\Gamma$, and let $S' \subset \pp^n$
  be the set of solutions of 
\begin{align*}
   &\left ( \ell_{j,0} x_0 + \cdots + \ell_{j,n} x_m = 0 \right )_{1 \le j \le n-i},\\    
  &    \left (\hat \ell_{j,0} F_0(x) + \cdots + \hat \ell_{j,m} F_m(x) = 0 \right )_{1 \le j \le i},\\
  &    x \notin V(F_0,\dots,F_m).
\end{align*}
 Then, all points $(x,y)$ in $S$ satisfy $x_0 F_0(x) y_0 \ne 0$ and
 $(x,y)\mapsto x$ gives a bijection $S \to S'$ with inverse
 $x \mapsto (x,\varphi(x))$.
\end{corr}
\begin{proof}
  Let $S$ be the fiber $\pi_2^{-1}(\ell,\hat \ell) \cap \Phi_\Gamma$,
  and let $S' \subset \pp^n$ be the points defined by the constraints
  above. Start first from a point $x$ in $S'$. Because $x$ is not in
  $V(F_0,\dots,F_m)$, we can define $y = \varphi(x)$, so that $(x,y)$
  is in $\Gamma$ (since we pointed out that $\Gamma$ can also be
  defined as the Zariski closure of the restriction of $\phi$ to
  $\pp^n - V(F_0,\dots,F_m)$). Then, $(x,y,\ell,\hat \ell)$ is in
  $\Phi_\Gamma$ and thus in $S$.  This gives an injection $a: S' \to
  S$.

  Conversely, start from a point $(x,y)$ in $S$. The previous lemma
  shows that $(x,y)$ is in $\Gamma_0$, so that $y =\varphi(x)$, which
  implies that $x$ satisfies all constraints defining $S'$. This gives
  a mapping $b:S \to S'$, such that $a \circ b: S \to S$ is the
  identity; as a result, $S$ and $S'$ are in one-to-one
  correspondence. We saw that all points $(x,y)$ in $S$ satisfy $x_0 F_0(x) y_0
  \ne 0$, so we are done.
\end{proof}

In order to bound the degree of the algebraic sets $\Phi_\Gamma$ and
$\Phi_\Sigma$, we need to give simple equations defining $\Gamma$ and
its tangent space at a point $(x,y)$. The definition of $\Gamma_0$
allows us to achieve this in the open set $\Omega$.
For this, we define the polynomials $G_i = Y_i F_0 - Y_0 F_i
\in k[X,Y]$, for $i=1,\dots,m$, together with their dehomogenization
$\tilde{G}_i = G_i(1,X_1,\dots,X_n,1,Y_1,\dots,Y_m)$. Similarly, given
$(x,y)$ in $\Omega$, we write $(\tilde x,\tilde
y)=(x_1/x_0,\dots,x_n/x_0,y_1/y_0,\dots,y_m/y_0)$ for the
corresponding point in $k^n\times k^m$.
\begin{lemma}\label{lemma:open}
  In the open set $\Omega$, $\Gamma$ coincides with
  $V(G_1,\dots,G_m)$; for any $(x,y)$ in $\Gamma \cap \Omega$, the
  tangent space $T_{(x,y)} \Gamma$ is the nullspace of the Jacobian
  matrix ${\rm Jac}(\tilde{G}_1,\dots,\tilde{G}_m)$ at $(\tilde x,\tilde y)$ and such
  a point is non-singular on $\Gamma$.
\end{lemma}
\begin{proof}
  Take $(x,y)$ in $\Gamma \cap \Omega$. Then, since $x_0 F_0(x)$ is
  non-zero, Lemma~\ref{lemma:gamma} shows that $y = \varphi(x)$, that
  is, $(y_0 : \cdots : y_m) = (F_0(x) : \cdots : F_m(x))$; thus, all
  $G_i$'s vanish at $(x,y)$.  Conversely, starting from $(x,y)$ in
  $V(G_1,\dots,G_m) \cap \Omega$, we see that the equalities
  $G_i(x,y)=0$ imply that $y_i = F_i(x)/F_0(x)$, as claimed.

  Consider such a point $(x,y)$, together with the corresponding
  affine point $(\tilde x, \tilde y)$. In the affine chart defined by
  $X_0=Y_0=1$, the Jacobian matrix of $\tilde G_1,\dots,\tilde G_m$
  with respect to $Y_1,\dots,Y_m$ at $(\tilde x,\tilde y)$ is a
  diagonal matrix, with $F_0(1,\tilde x)$ as diagonal. Since this value
  is non-zero, we deduce that the Jacobian of
  $\tilde{G}_1,\dots,\tilde{G}_m$ (with respect to all variables) 
  has full rank $m$ at $(\tilde x, \tilde y)$.
\end{proof}

From this, we can obtain useful degree bounds.

\begin{lemma}\label{lemma:PsiSigmaDegreeBound}
  The inequalities $\deg \left( \Psi_{\Gamma }\right) \leq
  2^{n}(d+1)^m$ and $\deg \left( \Psi_{\Sigma }\right) \leq
  2^{n}(d+1)^{m+1}$ hold.
\end{lemma}
\begin{proof}
  We give details of the proof for the case of $\Psi_\Sigma$, which is
  the (slightly) more involved; we comment below on the difference
  between the two cases. To begin with, remark that the previous
  lemma implies that $\Gamma \subset \pp^n \times \pp^m$ is one of the
  irreducible components of $V(G_1,\dots,G_m)$.

  Choose and fix a dehomogenization of $\pp^n \times \pp^m$, by
  setting $\nu(X)=\mu(Y)=1$, for some linear forms $\nu$ and $\mu$
  over respectively $\pp^n$ and $\pp^m$. These linear forms can be
  chosen arbitrarily, provided they satisfy constraints that will be
  made explicit below.

  Let us apply this dehomogenization to $\Gamma$ to obtain an affine
  set $\tilde \Gamma \subset k^n \times k^m$. The previous remark
  implies that $\tilde \Gamma$ is one of the irreducible components of
  $V(\tilde G_1,\dots,\tilde G_m)$, so by the affine B\'ezout
  inequality of Heintz~\cite{heintz1983definability}, it has degree at
  most $(d+1)^m$. We can apply the same dehomogenization to
  $\Gamma_\Sigma$, obtaining an affine set $\tilde \Gamma_\Sigma$,
  which coincides with $\tilde \Gamma \cap V(\tilde{F_0}) \subset k^n
  \times k^m$, where $\tilde{F_0}$ is the dehomogenized version of
  $X_0F_0$. Using the same B\'ezout inequality, we obtain $ \deg(\tilde
  \Gamma_\Sigma) \leq \deg(\tilde \Gamma) \deg(V(\tilde{F_0}) )\leq
  (d+1)^{m+1}.$

  We can similarly dehomogenize constructions in $\pp^n \times \pp^m
  \times k^{(n-i) (n+1)} \times k^{i (m+1)}$, in order to obtain
  objects in $k^n \times k^m \times k^{(n-i) (n+1)} \times k^{i
    (m+1)}$. We apply this process to $\Phi_{\Sigma}$ to obtain
  $\tilde \Phi_{\Sigma}$; this set can be rewritten as 
 $$\tilde \Phi_{\Sigma}= \tilde \Gamma_{\Sigma }\cap V(L(\tilde X))
  \cap V(\hat{L}(\tilde Y)) \subset k^n \times k^m \times k^{(n-i)
    \times (n+1)} \times k^{i \times (m+1)},$$ where $\tilde X,\tilde
  Y$ are obtained by dehomogenizing $X$ and $Y$ using the constraints
  $\nu(X)=1$ and $\mu(Y)=1$. Recall that we wrote $\Phi_\Sigma$ as the
  finite union of the irreducible sets $\Phi_1,\dots,\Phi_s$, with
  $\Phi_j = W_j \cap V(L(X)) \cap V(\hat L(Y))$; our assumption on the
  linear forms $\nu$ and $\mu$ is then that $\nu(X)\mu(Y)$ vanishes
  identically on none of $\Phi_1,\dots,\Phi_s$.

  Independently of this assumption, we always have that $\tilde
  \Phi_\Sigma$ is the union of the sets $\tilde \Phi_1,\dots,
  \tilde \Phi_s$ obtained by dehomogenizing $\Phi_1,\dots,\Phi_s$;
  our assumption implies that each $\tilde \Phi_j$ is open dense 
  in $\Phi_j$.

  In terms of degree, the equations $L(\tilde X)$ and $\hat L(\tilde
  Y)$ have degree $2$, and there are $n$ of them. Again applying the
  affine B\'ezout inequality of Heintz~\cite{heintz1983definability},
  we deduce that $$ \deg(\tilde \Phi_{\Sigma}) \leq \deg(\tilde
  \Gamma_{\Sigma}) \deg( V(L(\tilde X))) \deg(V(\hat{L}(\tilde
  Y)))\leq 2^n (d+1)^{m+1}.$$
  
  Finally, consider the projection on the $L,\hat L$-space $k^{(n-i)
    (n+1)} \times k^{i (m+1)}$; in a slight abuse of notation, we
  denote by $\pi_2$ both projections $\pp^n\times \pp^m \times
  k^{(n-i) (n+1)} \times k^{i (m+1)} \to k^{(n-i) (n+1)} \times k^{i
    (m+1)}$ and $k^n\times k^m \times k^{(n-i) (n+1)} \times k^{i
    (m+1)} \to k^{(n-i) (n+1)} \times k^{i (m+1)}$. In an affine
  setting, we know that the degree cannot increase through a
  projection, so that we have
  $$ \deg \left (\overline {\pi_2\left (\tilde \Phi_\Sigma \right )}\right ) \le  2^n (d+1)^{m+1}.$$
  Thus, to conclude, it is enough to prove that 
  $\pi_2(\Phi_\Sigma) = \overline {\pi_2\left (\tilde \Phi_\Sigma
    \right )}.$ By construction, the right-hand side is the union of
  the sets $\overline{\pi_2(\tilde
    \Phi_1)},\dots,\overline{\pi_2(\tilde \Phi_s)}$.  Now, by
  assumption, each $\tilde \Phi_j$ is an open dense subset of the
  corresponding irreducible set $\Phi_j$; this implies that
  $\overline{\pi_2(\tilde \Phi_j)}=\pi_2(\Phi_j)$, and the conclusion
  follows.

  In the case of $\Psi_\Gamma$, essentially the same construction
  applies, but we would not need to intersect with $V(\tilde F_0)$;
  this saves a factor $(d+1)$, for a total of $2^n (d+1)^m$.
\end{proof}
 
If the restriction of $\pi_2$ to $\Phi_\Gamma$ is not surjective, we are
now done.

\begin{proof}[Proof of
    Proposition~\ref{propn:projectiveDegreeExplicit}, non-surjective
    case.]  Take $(\ell,\hat \ell)$ not in $\Phi_\Gamma$. By
  Corollary~\ref{coro:oneone}, the system of equations and inequations
  in Proposition~\ref{propn:projectiveDegreeExplicit} has no solution.
  If the restriction of $\pi_2$ to $\Phi_\Gamma$ is not surjective,
  because $\pi_2$ is closed, $\Psi_\Gamma$ must have codimension at
  least $1$.  Thus, we take $g_i=0$ and let $F$ be a non-zero
  polynomial whose zero-set contains $\Psi_\Gamma$. The degree bound
  of the previous lemma allows us to conclude.
\end{proof}

In what follows, we can thus assume that $\Psi_\Gamma = k^{(n-i)
  (n+1)} \times k^{i (m+1)}$; in particular, this gives us an
extension $k(L,\hat L) \to k(\Phi_\Gamma)$ of finite degree $D$.

\begin{lemma}
  The degree $D=[k(\Phi_\Gamma) : k(L,\hat L)]$ is at most $d^n$.
\end{lemma}
\begin{proof}
  Although it was proved for $(\ell,\hat\ell)$ with coefficients in $k$,
  Corollary~\ref{coro:oneone} also applies to indeterminates $(L,\hat L)$,
  since by construction they do not satisfy the equations defining 
  $\Psi_\Sigma$. Thus, the  
  the generic fiber $\pi_2^{-1}(L,\hat L) \cap \Phi_\Gamma$ is
in one-to-one correspondence with the set $S'$ of solutions of 
\begin{align*}
   &\left ( L_{j,0} x_0 + \cdots + L_{j,n} x_m = 0 \right )_{1 \le j \le n-i},\\    
  &    \left (\hat L_{j,0} F_0(x) + \cdots + \hat L_{j,m} F_m(x) = 0 \right )_{1 \le j \le i},\\
  &    x \notin V(F_0,\dots,F_m),
\end{align*}
  where the mapping from the latter to the former is $x \mapsto
  (x,\varphi(x))$.  Let $K$ be the function field of $S'$ over
  $k(L,\hat L)$; this then is a finite extension of $k(L,\hat L)$ of degree
  $D$ as well.
  
  Now, working over an algebraic closure of $k(L,\hat L)$, we see that
  this degree is bounded from above by the sum of the multiplicities
  of the isolated solutions of the system
\begin{align*}
   &\left ( L_{j,0} x_0 + \cdots + L_{j,n} x_m = 0 \right )_{1 \le j \le n-i},\\    
  &    \left (\hat L_{j,0} F_0(x) + \cdots + \hat L_{j,m} F_m(x) = 0 \right )_{1 \le j \le i};
\end{align*}
 by B\'ezout's theorem, this is at most $d^n$.
\end{proof}

To summarize, applying Lemma~\ref{lemma:fiber} to $\pi_2$, we have
that all fibers of its regular values have the same cardinality, which
we call $g_i$. On the other hand, the previous lemma, together with
our assumption on the characteristic of $k$, allow us to apply Sard's
Lemma; this implies that the regular values contain an open dense
subset of $k^{(n-i) (n+1)} \times k^{i (m+1)}$. It remains to give
sufficient conditions on a pair $(\ell,\hat \ell) \in k^{(n-i) (n+1)}
\times k^{i (m+1)}$ that will ensure that $(\ell,\hat \ell)$ is such a
regular value.

\begin{lemma}
  Let $(\ell,\hat \ell)$ be in $k^{(n-i) (n+1)} \times k^{i (m+1)} -
  \Psi_\Sigma$. Then, $(\ell,\hat \ell)$ is a regular value of the
  restriction of $\pi_2$ to $\Gamma$ if and only if for any $(x,y)$ in
  the fiber $\pi_2^{-1}(x,y,\ell,\hat \ell) \cap \Phi_\Gamma$, the
  $(m+n) \times (m+n)$ matrix
  $$J^\star= \begin{bmatrix} 
     {\rm Jac}(\tilde{G}_1,\dots,\tilde{G}_m) \\
      \begin{array}{cc}
       \ell^*~~  & ~~0
       \end{array}  \\
     \begin{array}{lr}
       0 ~~& ~~\hat{\ell}^*
       \end{array} 
  \end{bmatrix}$$
    is invertible at $(\tilde x, \tilde y, \ell, \hat \ell)$, where 
    $\ell^*$ and $\hat \ell^*$ denote the matrices
$$ \ell^* = \begin{bmatrix}
      \ell_{1,1} & \cdots & \ell_{1,n} \\
        \vdots & & \vdots \\
        \ell_{n-i,1} & \cdots & \ell_{n-i,n}
    \end{bmatrix} \quad\text{and}\quad
\hat \ell^* = \begin{bmatrix}
    \hat \ell_{1,1} & \cdots & \hat \ell_{1,m} \\
        \vdots & & \vdots \\
     \hat \ell_{i,1} & \cdots & \hat \ell_{i,m}
    \end{bmatrix}.$$
\label{lemma:MainLemma}
\end{lemma}
\begin{proof}
  Lemma~\ref{lemma:open} implies that in the open set $\Omega \times
  k^{(n-i) (n+1)} \times k^{i (m+1)}$, $\Phi_\Gamma$ is defined by the
  equations $G_1=\dots=G_m=0$, together with $L(X) =\hat L(Y) =
  0$. Let us work in the affine chart defined by $X_0=Y_0=1$; in this
  chart, the equations for $\Phi_\Gamma$ become $\tilde{G}_1=\dots=\tilde{G}_m=0$,
  together with $L(1,\tilde X)=\hat L(1,\tilde Y)$, with $\tilde X =
  (X_1,\dots,X_n)$ and $\tilde Y = (Y_1,\dots,Y_m)$.  Consider the
  Jacobian matrix of these polynomials with respect to variables
  $\tilde X, \tilde Y, L, \hat L$ (in this order); this matrix 
  takes the form
    $$J= \begin{bmatrix} 
     {\rm Jac}(\tilde{G}_1,\dots,\tilde{G}_m) & 0 & 0 \\
       \begin{array}{cc}
       L^\star~~ & ~~0
       \end{array}  & D_L & 0 \\
       \begin{array}{cc}
       0 ~~& ~~\hat{L}^\star
       \end{array} & 0 & D_{\hat L}
  \end{bmatrix},$$
  where 
  \begin{itemize}
  \item $L^\star$ and $\hat L^\star$ are defined similarly to
    $\ell^\star$ and $\hat \ell^\star$, but with indeterminate entries,
  \item $D_L$ and $D_{\hat L}$ are full-rank matrices (since they possess a
  maximal identity submatrix, coming from the derivatives with respect
  to variables $L_{j,0}$ and $\hat L_{j,0}$).
  \end{itemize}
 Since we saw in Lemma~\ref{lemma:open}
  that ${\rm Jac}(\tilde{G}_1,\dots,\tilde{G}_m)$ has full rank $m$ at
  $(\tilde x,\tilde y)$, $J$ has full rank $m + n$ at $(\tilde
  x,\tilde y,\ell,\hat \ell)$.  As a consequence, $(x,y,\ell,\hat
  \ell)$ is non-singular on $\Phi_\Gamma$, and the nullspace of
  $J(x,y,\ell,\hat \ell)$ defines the tangent space $T_{(x,y,\ell,\hat
    \ell)} \Phi_\Gamma$. From this, one easily deduces that
  $\pi_2(T_{(x,y,\ell,\hat \ell)} \Phi_\Gamma)$ has dimension less
  than $(n-i) (n+1)+i (m+1)$ if and only if the
  submatrix 
  \begin{equation}\label{eq:Jstar}
J^\star= \begin{bmatrix} {\rm
      Jac}(\tilde{G}_1,\dots,\tilde{G}_m) \\
      \begin{array}{cc}
       L^\star~~ & ~~0
       \end{array}  \\
     \begin{array}{cc}
       0~~ & ~~\hat{L}^\star
       \end{array}
  \end{bmatrix}    
  \end{equation}
  has rank less than $m+n$.
\end{proof}

Let us denote by $\hat \Phi_\Gamma \subset k^n \times k^m \times
k^{(n-i) (n+1)} \times k^{i (m+1)}$ the dehomogenization of
$\Phi_\Gamma$ obtained by setting $X_0=Y_0=1$ (this should not be
confused with the dehomogenizations defined in the proof of
Lemma~\ref{lemma:PsiSigmaDegreeBound}, that were obtained using random
linear forms). Remark that $\hat \Phi_\Gamma$ is not empty: the first
item in the previous lemma shows that above a generic choice of
$(\ell,\hat \ell)$, there exist points in $\Phi_\Gamma \cap \Omega$,
so that $X_0 Y_0$ does not vanish identically on $\Phi_\Gamma$.  Thus,
$\hat \Phi_\Gamma$ is an irreducible affine algebraic set; the degree
inequality established in the proof of
Lemma~\ref{lemma:PsiSigmaDegreeBound} for generic dehomogenizations
still holds, giving $\deg(\hat \Phi_\Gamma) \le 2^n (d+1)^m$.

Let $D$ be the determinant of the matrix $J^\star$ defined
in~\eqref{eq:Jstar}, and define $\Delta = \hat \Phi_\Gamma \cap V(D)
\subset k^n \times k^m \times k^{(n-i) (n+1)} \times k^{i (m+1)}$.
Finally, we let $\Delta'\subset k^n \times k^m \times k^{(n-i) (n+1)}
\times k^{i (m+1)}$ be the Zariski closure of $\pi_2(\Delta)$ (where
we see $\pi_2$ as the projection $k^n \times k^m \times k^{(n-i)
  (n+1)} \times k^{i (m+1)}\to k^{(n-i) (n+1)} \times k^{i (m+1)}$).

\begin{lemma}The following holds:
  \begin{itemize}
  \item Any $(\ell,\hat \ell)$ in $k^{(n-i) (n+1)} \times k^{i (m+1)}
    - \Psi_\Sigma - \Delta'$ is a regular value of the restriction of
    $\pi_2$ to $\Gamma$.
  \item The algebraic set $\Delta'$ has codimension at least 1, and degree
    at most $(n+m) 2^n(d+1)^{m+1}$.
  \end{itemize} \label{lemma:deltaLemma}
\end{lemma}
\begin{proof}
  For $(\ell,\hat \ell)$ in $k^{(n-i) (n+1)} \times k^{i (m+1)} -
  \Psi_\Sigma$, we claim that $(\ell,\hat \ell)$ is a critical value
  of the restriction of $\pi_2$ to $\Phi_\Gamma$ if and only if
  $(\ell,\hat\ell)$ belongs to $\pi_2(\Delta)$.

  Indeed, take $(\ell,\hat \ell)$ in $k^{(n-i) (n+1)} \times k^{i
    (m+1)} - \Psi_\Sigma$ and a point $(x,y,\ell,\hat\ell)$ in
  $\Phi_\Gamma$ lying over $(\ell,\hat \ell)$. Then, the previous
  lemma shows that $x_0 y_0$ is non-zero, so that $(\tilde x,\tilde
  y,\ell,\hat\ell)$ is well-defined, and belongs to $\hat
  \Phi_\Gamma$. This shows that $(\tilde x,\tilde y,\ell,\hat\ell)$ is
  well-defined, and belongs to $\Delta$ if and only if it belongs to
  $V(D)$. The claim then follows from the second item in the previous
  lemma.

  This claim directly implies the first statement in the lemma.  By
  means of Sard's Lemma, the claim implies as well that the projection
  $\pi_2(\Delta)$ is contained in a strict algebraic subset of
  $k^{(n-i) (n+1)} \times k^{i (m+1)}$, so that its Zariski closure
  $\Delta'$ has codimension at least 1.

  Finally, to give an upper bound on the degree of $\Delta'$, it is
  enough to bound the degree of $\Delta$. We know that $\deg(\hat
  \Phi_\Gamma) \le 2^n(d+1)^m$, and the determinant $D$ has degree at
  most $md+n \le (m+n)(d+1)$. The desired bound is yet another
  consequence of Heintz's affine B\'ezout inequality
  \cite{heintz1983definability}.
\end{proof}
We can now conclude the proof of Proposition~\ref{propn:projectiveDegreeExplicit}. 
\begin{proof}[Proof of Proposition~\ref{propn:projectiveDegreeExplicit}, surjective case.]
We define $F$ as
$F=F_1 F_2$, where
\begin{itemize}
\item $F_1 \in k[L,\hat L]$ is any non-zero polynomial of degree at most
  $2^{n}(d+1)^{m+1}$ such that $V(F_1)$ contains $\Psi_\Sigma$ (such a
  polynomial exists, in view of
  Lemma~\ref{lemma:PsiSigmaDegreeBound});
\item $F_2 \in k[L,\hat L]$ is is any non-zero polynomial of degree at most
  $(m+n)2^n(d+1)^{m+1}$ such that $V(F_2)$ contains $\Delta'$ (such a
  polynomial exists, in view of the previous lemma, Lemma \ref{lemma:deltaLemma}).
\end{itemize}
Thus, $F$ has degree at most $(m+n+1)2^n (d+1)^{m+1}$.  The previous
lemma shows that any $(\ell,\hat \ell)$ for which $F$ is nonzero is a
regular value of the restriction of $\pi_2$ to $\Phi_\Gamma$, so by
Lemma~\ref{lemma:fiber}, the fiber $S=\pi_2^{-1}(\ell,\hat \ell) \cap
\Phi_\Gamma$ has cardinality $g_i$; we conclude using Corollary~\ref{coro:oneone}.
\end{proof}

\end{appendices}
\newpage
\bibliographystyle{plain}
\bibliography{refs}
\end{document}